\documentclass[9pt]{amsart}
\usepackage{amsmath,amsfonts,amsthm}
\usepackage{color}
\usepackage{verbatim}
\usepackage{eepic,epic}
\usepackage{epsfig,subfigure,epstopdf}
\usepackage[]{graphics}
\usepackage[colorlinks,linktocpage,linkcolor=blue]{hyperref}

\newtheorem{theorem}{Theorem}[section]

\newtheorem{lemma}{Lemma}[section]

\newtheorem{assumption}[theorem]{Assumption}
\newtheorem{remark}{Remark}[section]

\numberwithin{equation}{section}

\newcommand{\ga}{\gamma}
\newcommand{\Ga}{\Gamma}

\newcommand{\fy}{\varphi}
\newcommand{\la}{\lambda}

\def\al{{\alpha}}
\def\RR{{\mathbb R}}
\def\II{{(D)}}

\newcount\icount

\def\DD#1#2{\icount=#1
  \ifnum\icount<1
  \,_{ 0}\kern -.1em D^{#2}_{\kern -.1em x}
  \else
  \,_{x}\kern -.2em D^{#2}_1
  \fi
}

\def\DDRI#1#2{\icount=#1
  \ifnum\icount<1
  \,_{-\infty}^{\kern 1em R}\kern -.2em D^{#2}_{\kern -.1em x}
  \else
  \,_{x}^R \kern -.2em D^{#2}_\infty
  \fi
}

\def\DDR#1#2{\icount=#1
  \ifnum\icount<1
 _{0}^{ \kern -.1em R} \kern -.2em D^{#2}_{\kern -.1em x}
  \else
 _{x}^{ \kern -.1em R} \kern -.2em D^{#2}_{\kern -.1em 1}
  \fi
}

\def\DDCI#1#2{\icount=#1
  \ifnum\icount<1
  \,_{-\infty}^{\kern 1em C}  \kern -.2em D^{#2}_{\kern -.1em x}
  \else
  \,_{x}^C \kern -.2em  D^{#2}_\infty
  \fi
}

\def\DDC#1#2{\icount=#1
  \ifnum\icount<1
  \,_{0}^C \kern -.2em  D^{#2}_{\kern -.1em x}
  \else
  \,_{x}^C \kern -.2em D^{#2}_1
  \fi
}

\def\Hd#1{\widetilde H^{#1}(D)}

\def\Hdi#1#2{\icount=#1
  \ifnum\icount<1
  \widetilde H_{L}^{#2}\II
  \else
  \widetilde H_{R}^{#2}\II
  \fi
}

\begin{document}
\title[FEM with reconstruction for fractional BVP]
{A Simple Finite Element Method for Boundary Value Problems with a Riemann-Liouville Derivative}
\author {Bangti Jin \and Raytcho Lazarov \and Xiliang Lu \and Zhi Zhou}
\address {Department of Computer Science, University College London, Gower Street, London WC1E 6BT, UK
({bangti.jin@gmail.com})}
\address {Department of Mathematics, Texas A\&M University, College Station, TX 77843-3368
({{lazarov@math.tamu.edu}, {zzhou@math.tamu.edu}})}
\address {School of Mathematics and Statistics, Wuhan University, Wuhan 430072,
P.R. China. ({{xllv.math@whu.edu.cn}})}

\date{started June, 2014; today is \today}

\begin{abstract}
We consider a boundary value problem involving a Riemann-Liouville fractional derivative
of order $\al\in (3/2,2)$ on the unit interval $(0,1)$. The standard Galerkin finite
element approximation converges slowly due to the presence of singularity term $x^{\al-1}$
in the solution representation. In this work, we develop a simple technique, by transforming
it into a second-order two-point boundary value problem with nonlocal low order terms, whose solution
can reconstruct directly the solution to the original problem. The stability of the
variational formulation, and the optimal regularity pickup of the solution are analyzed.
A novel Galerkin finite element method with piecewise linear or quadratic finite elements
is developed, and $L^2\II$ error estimates are provided. The approach is then applied to the
corresponding fractional Sturm-Liouville problem, and error estimates of the eigenvalue
approximations are given. Extensive numerical results fully confirm our theoretical study.
\end{abstract}

\maketitle

\section{Introduction}\label{sec:intro}

In this work, we consider the following boundary value problem involving a Riemann-Liouville fractional derivative
\begin{equation}\label{eqn:fde}
  \begin{aligned}
    -\DDR 0 \alpha u + qu &= f\quad \mbox {in } D \equiv (0,1),\\
     u(0)=u(1) &= 0,
  \end{aligned}
\end{equation}
where $f\in L^2(D)$, and $\DDR 0 \alpha u $ denotes the Riemann-Liouville fractional derivative of order
$\alpha\in(3/2,2)$, defined in \eqref{eqn:Riemann} below. The choice $\alpha\in(3/2,2)$ is mainly technical,
since for $\alpha\in(1,3/2]$, the analysis below does not carry over, even though numerically the technique
to be developed works well. For $\alpha=2$, the fractional derivative $\DDR
0 \alpha u$ recovers the usual second-order derivative $u''$, and thus the model \eqref{eqn:fde} can
be viewed as the fractional counterpart of the classical two-point boundary value problem.

Problem \eqref{eqn:fde} arises in the mathematical modeling of superdiffusion process in heterogeneous media, in which the
mean square variance grows faster than that in the Gaussian process. It has found applications in
magnetized plasma \cite{delCastillo:2003,delCastillo:2005} and subsurface flow \cite{BensonWheatcraftMeerschaert:2000}. The numerical study of problem \eqref{eqn:fde} is quite
extensive. Among existing methods, the finite difference method based on the shifted Gr\"unwald-Letnikov
formula is predominant, since the earlier introduction \cite{TadjeranMeerschaertScheffler:2006}; and
see also \cite{BaeumerKovacs:2014} for higher order schemes. However, in these interesting
works, one standing assumption is that the solution is sufficiently smooth, which unfortunately is
generally not justified \cite{JinLazarovPasciak:2013a}. To this date, the precise condition under which
the solution to \eqref{eqn:fde} is indeed smooth remains unclear. Recently, finite element methods
(FEMs) \cite{ErvinRoop:2006,WangYang:2013} were developed and analyzed.

One of the main challenges in accurately solving problem \eqref{eqn:fde} is that the solution contains a
singular term $x^{\alpha-1}$ (see \cite{JinLazarovPasciak:2013a} and Section \ref{sec:prelim} below), which
in turn limits the global solution regularity and thus also the accuracy of numerical approximations. One way to resolve the issue
is the singularity reconstruction technique recently developed by the first and fourth named authors \cite{JinZhou:2014} and inspired
by \cite{CaiKim:2001}, in which the solution is split into a singular part containing the term $x^{\alpha-1}$,
and a regular part. A variational formulation of the regular part is derived, and the singularity strength
is then reconstructed from the regular part. The numerical experiments in \cite{JinZhou:2014} indicate that the method converges
well for problem \eqref{eqn:fde}, with provable $L^2\II$ convergence rates, which improves that for the
standard Galerkin FEM. However, the extension of the method to the related Sturm-Liouville problem seems
not viable, due to the nonlinear nature of the eigenvalue problem.

In this work, we develop a novel approach for solving problem \eqref{eqn:fde} based on transformation. It
retains the salient features of the singularity reconstruction approach, i.e., resolving accurately
the singularity, enhanced convergence rates and easy implementation. Meanwhile it can be extended straightforwardly to the
related Sturm-Liouville problem with a Riemann-Liouville fractional derivative in the leading term, and the resulting linear system can
be solved efficiently by a preconditioning technique. The approach is motivated by the following observation: under
the Riemann-Liouville integral transformation ${_0\hspace{-0.3mm}I_x^{2-\alpha} u}$,
cf. \eqref{eqn:RL-int}, the leading singularity $x^{\alpha-1}$ is actually smoothed into a very smooth function $x$,
which can be well approximated by the standard conforming finite elements or orthogonal polynomials. We
shall derive a new formulation for the transformed variable, and analyze its stability and the finite element
approximation. Further, the approach is extended to the related Sturm-Liouville problem, and the convergence rate
is also established.

The rest of the paper is organized as follows. In Section \ref{sec:prelim} we recall preliminaries of fractional
calculus, including properties of fractional integral and differential operators in Sobolev spaces. Then in
Section \ref{sec:soltheory}, we derive the new approach, develop the proper variational formulation, and establish
stability estimates. The Galerkin FEM with continuous piecewise linear and quadratic
finite elements is discussed in Section \ref{sec:fem}. $L^2\II$ error estimates are provided for the FEM
approximations to \eqref{eqn:fde}. The approach is then extended to the Sturm-Liouville problem in Section
\ref{sec:eig}. Finally, extensive numerical results are presented in Section \ref{sec:numer} to verify the efficiency and
accuracy of the new approach. Throughout, the notation $c$, with or without a subscript, denote a generic constant,
which may differ at different occurrences, but it is always independent of the mesh size $h$.

\section{Preliminaries}\label{sec:prelim}

We first recall the definition of the Riemann-Liouville fractional derivative. For any $\beta>0$ with $n-1 < \beta < n$,
$n\in \mathbb{N}$, the left-sided Riemann-Liouville fractional derivative $\DDR0\beta u$ of order $\beta$ of a function
$u\in C^n[0,1]$ is defined by \cite[pp. 70]{KilbasSrivastavaTrujillo:2006}:
\begin{equation}\label{eqn:Riemann}
  \DDR0\beta u =\frac {d^n} {d x^n} \bigg({_0\hspace{-0.3mm}I^{n-\beta}_x} u\bigg) .
\end{equation}
Here $_0\hspace{-0.3mm}I^{\gamma}_x$ for $\gamma>0$ is the left-sided Riemann-Liouville fractional integral operator
of order $\gamma$ defined by
\begin{equation}\label{eqn:RL-int}
 ({\,_0\hspace{-0.3mm}I^\gamma_x} f) (x)= \frac 1{\Gamma(\gamma)} \int_0^x (x-t)^{\gamma-1} f(t)dt,
\end{equation}
where $\Gamma(\cdot)$ is Euler's Gamma function defined by $\Gamma(x)=\int_0^\infty t^{x-1}e^{-t}dt$. The right-sided
versions of the fractional-order integral operator $_xI_1^\gamma$ and derivative operator $\DDR 1 \beta$ are defined analogously by
\begin{equation*}
  ({_x\hspace{-0.3mm}I^\gamma_1} f) (x)= \frac 1{\Gamma(\gamma)}\int_x^1 (t-x)^{\gamma-1}f(t)\,dt\quad\mbox{and}\quad
  \DDR1\beta u =(-1)^n\frac {d^n} {d x^n} \bigg({_x\hspace{-0.3mm}I^{n-\beta}_1} u\bigg) .
\end{equation*}

Now we introduce some function spaces. For any $\beta\ge 0$, we denote $H^\beta\II$ to be
the Sobolev space of order $\beta$ on the unit interval $D$, and $\Hd \beta $ to be the set of
functions in $H^\beta\II$ whose extension by zero to $\RR$ are in $H^\beta(\RR)$. Analogously, we define
$\Hdi 0 \beta$ (respectively, $\Hdi 1 \beta$) to be the set of functions $u$ whose extension by zero, denoted by
$\tilde{u}$, is in $H^\beta(-\infty,1)$ (respectively, $H^\beta(0,\infty)$). For $u\in \Hdi 0
\beta$, we set $\|u\|_{\Hdi 0\beta}:=\|\tilde{u}\|_{H^\beta(-\infty,1)}$, and analogously the norm
in $\Hdi 1 \beta$.

The following theorem collects their important properties \cite[pp. 73, Lemma 2.3]{KilbasSrivastavaTrujillo:2006}
\cite[Theorems 2.1 and 3.1]{JinLazarovPasciak:2013a}. In particular, Theorem \ref{thm:fracop}(b) extends the
domain of the operator $\DDR 0 \beta$ from $C^n[0,1]$ to $\Hdi 0\beta $.
\begin{theorem}\label{thm:fracop}
The following statements hold.
\begin{itemize}
  \item[$\mathrm{(a)}$] The integral operators $_0I_x^\beta$ and $_xI_1^\beta$ satisfy the semigroup property.
  \item[$\mathrm{(b)}$] The operators $\DDR0\beta$ and $\DDR1\beta $ extend continuously to operators
     from $\Hdi 0 \beta$ and $\Hdi 1\beta$, respectively, to $L^2\II$.
  \item[$\mathrm{(c)}$] For any $s,\beta\geq 0$, the operator
$_0I_x^\beta$ is bounded from $\Hdi0 s$ to $\Hdi0{\beta+s}$,
and $_xI_1^\beta$ is bounded from $\Hdi1 s$ to $\Hdi1{\beta+s}$.
\end{itemize}
\end{theorem}

We shall also need an ``algebraic'' property of the space $\Hd s$, $0<s<1$
\cite[Lemma 4.6]{JinLazarovPasciak:2013a}.
\begin{lemma}\label{lem:Hmulti}
Let $0<s\leq1$, $s\neq 1/2$. Then for any $u\in H^ s\II \cap L^\infty\II$
and $v\in \Hd s\cap L^\infty\II$, $uv\in \Hd s$.
\end{lemma}

Now we describe the variational formulation. We first introduce the bilinear form
\begin{equation}
a(u,v) = -(\DDR0{{\alpha/2}}u,\,\DDR1{{\alpha/2}}v) +(qu,v).
\end{equation}
Then the variational formulation for problem \eqref{eqn:fde} is given
by: find $u\in \Hd {\alpha/2}$ such that
\begin{equation}\label{varrl}
    a(u,v)=(f,v)\quad \forall v\in \Hd {\alpha/2}.
\end{equation}
For trivial case $q\equiv0$, the well-posedness follows from the boundedness and coercivity of $-(\DDR0{{\alpha/2}}
\cdot,\,\DDR1{{\alpha/2}}\cdot)$ in $\Hd{\al/2}$ (see \cite[Lemma 3.1]{ErvinRoop:2006}, \cite[Lemma
4.2]{JinLazarovPasciak:2013a}). Simple computation shows that the variational solution $u$
of \eqref{varrl} is given by
\begin{equation}\label{eqn:sols}
u(x)=-({_0I_x^\al}f)(x)+({_0I_x^\al}f)(1)x^{\al-1},
\end{equation}
and it satisfies the strong formulation \eqref{eqn:fde}.

To study the bilinear form $a(\cdot,\cdot)$ in general case, i.e. $q\neq0$, we make the following assumption.
\begin{assumption} \label{ass:riem0}
Let the bilinear form $a(u,v)$ with $u,v\in \Hd{\al/2}$ satisfy
\begin{itemize}
 \item[{$\mathrm{(a)}$}]  The problem of finding $u \in \Hd{\al/2}$ such that $a(u,v)=0$ for all $v \in \Hd{\al/2}$
           has only the trivial solution $u\equiv 0$.
 \item[{$(\mathrm{a}^\ast)$}] The problem of finding $v \in \Hd{\al/2}$ such that $a(u,v)=0$ for all $u \in \Hd{\al/2}$
    has only the trivial solution $v\equiv 0$.
\end{itemize}
\end{assumption}

Under Assumption \ref{ass:riem0}, there exists a unique solution $u\in \Hd{\al/2}$
to \eqref{varrl} \cite[Theorem 4.3]{JinLazarovPasciak:2013a}. In fact the variational
solution is a strong solution. To see this, we consider the problem
$ 
-\DDR0{\alpha} u = f-qu.
$ 
A strong solution is given by \eqref{eqn:sols} with a right hand side $\widetilde{f}=f-qu$.
It satisfies the variational equation \eqref{varrl} and hence coincides with the unique variational
solution. Further, the solution $u$ satisfies the stability estimate
$\|u \|_{\Hdi 0 {\al -1 +\beta}} \le c\|f\|_{L^2\II},$
for any $\beta\in(2-\alpha,1/2)$.
The representation \eqref{eqn:sols} indicates that the global regularity of the solution
$u$ does not improve with the regularity of the source term $f$, due to the inherent
presence of the term $x^{\alpha-1}$.

\section{A new approach: Variational formulation and regularity}\label{sec:soltheory}
In this section, we develop a new approach for problem \eqref{eqn:fde}. We first motivate the
approach, and then discuss the variational stability and regularity pickup. The adjoint problem
is also briefly discussed.
\subsection{Motivation of the new approach}
First, we motivate the new approach. The basic idea is to absorb the leading singularity $x^{\alpha-1}$ into the
problem formulation. To this end, we set
\begin{equation}\label{eqn:reconstruct_u}
u={\DDR 0 {2-\alpha}} w - ({\DDR 0 {2-\alpha}}w)(1)x^\mu,
\end{equation}
where $\mu\geq \alpha$ is a parameter to be selected. The motivation behind the choice of the fractional
derivative $\DDR 0 {2-\alpha} w$ is that the primitive of the singularity
$x^{\alpha-1}$ under the ``fractional'' transformation is $x$ (up to a multiplicative constant), which is
smooth and can be accurately approximated by standard finite element functions. The second term in the expression
is to keep the boundary condition $u(1)=0$. From the condition $w(0)=0$, we deduce that $u(0)=0$
(for more details see the proof of Theorem \ref{thm:rec}). Upon
substituting it back into \eqref{eqn:fde}, and noting that for $w\in \Hd1$
\begin{equation*}
 \left( {\,_0\hspace{-0.3mm}I^{\al-1}_x}w\right)'(x) = \left( {\,_0\hspace{-0.3mm}I^{\al-1}_x}w'\right)(x),
\end{equation*}
we arrive at
\begin{equation}\label{eqn:deriv}
  \begin{aligned}
    -\DDR 0 \alpha u + q u 
     & = - w'' + ({\DDR0 {2-\alpha}}w)(1)(c_0x^{\mu-\alpha}-qx^\mu) +
     q\, {\DDR 0 {2-\alpha}} w, 
  \end{aligned}
\end{equation}
where the constant $c_0$ is defined as
\begin{equation}\label{c-zero}
c_0={\Gamma(\mu+1)}/{\Gamma(1+\mu-\al)}.
\end{equation}
Here the second line follows from
the boundary condition $w(0)=0$ and the identity
\begin{equation*}
   {\DDR 0 \alpha}\,{\DDR 0 {2-\alpha}} w = ({_0I_x^{2-\alpha}} \DDR 0 {2-\alpha} w)'' = w''.
\end{equation*}
Consequently, the transformed variable $w$ solves the boundary value problem
\begin{equation}\label{eqn:fde-new}
  \begin{aligned}
  -w'' + q\,{\DDR 0 {2-\alpha}} w &+ ({\DDR 0 {2-\alpha}}w)(1)
     \left(c_0 x^{\mu-\al}-qx^\mu\right) = f  \quad \mbox{ in } D, \\
      w(0)&=w(1) = 0.
  \end{aligned}
\end{equation}

Once problem \eqref{eqn:fde-new} is solved, the solution $u$ to problem \eqref{eqn:fde} can be reconstructed
from \eqref{eqn:reconstruct_u}. Equation \eqref{eqn:fde-new}
is a boundary value problem for an integro-differential equation and
has a number of distinct features:
\begin{itemize}
  \item[(a)] The leading term involves a canonical second-order derivative, and thus the solution
     $w$ is free from singularity, if the source term $f$ is smooth. This overcomes one of the
     main challenges inherent to the fractional formulation \eqref{eqn:fde}.
  \item[(b)] In the resulting linear system from the Galerkin discretization of problem \eqref{eqn:fde-new},
     the leading term is dominant and has a simple structure; it can naturally
     act as a preconditioner.
  \item[(c)] The approach extends straightforwardly to the related Sturm-Liouville problem of finding
     the eigenpairs.
\end{itemize}

\begin{remark}\label{rem:reg-mu2}
Throughout, the condition $\mu\geq \alpha$ will be assumed below. Note that the choice
$\mu=\alpha-1$ is also of special interest, for which, with the identity $\DDR 0 \alpha x^{\alpha-1}
= 0 $, the modified equation reads
\begin{equation*}
  \begin{aligned}
    - w'' + q(x)\,\, {\DDR 0 {2-\alpha}} w - ({\DDR 0 {2-\alpha}}w)(1) \, q(x) \,x^{\alpha-1} & = f(x)
          \quad \mbox{in } D,\\
     w(0) = w(1) &=0.
  \end{aligned}
\end{equation*}
Since $\alpha>3/2$, the term $x^{\alpha-1}$ belongs to the space $H^1\II$. Thus, the theoretical developments
below, especially Theorem \ref{thm:reg-new}, remain valid for this choice.
\end{remark}
\subsection{Variational stability}

Next we discuss the well-posedness of the formulation \eqref{eqn:fde-new} for the case $\alpha\in(3/2,2)$, by showing
\begin{itemize}
\item[(a)] Problem \eqref{eqn:fde-new} has a unique solution $w \in \Hd 1$ and certain regularity pickup;
\item[(b)] $u={\DDR 0 {2-\alpha}} w-({\DDR 0 {2-\alpha}}w)(1) x^{\mu}$
is the solution of problem \eqref{eqn:fde}.
\end{itemize}
Further, we shall consider the following general problem: For $\al\in (3/2,2)$, find $w$
\begin{equation}\label{eqn:fde2}
  \begin{aligned}
  -w'' + q\,\, {\DDR 0 {2-\alpha}} w &+ p \, ({\DDR 0 {2-\alpha}}w)(1)\, = f \quad \mbox{ in } D, \\
      w(0)&=w(1) = 0,
  \end{aligned}
\end{equation}
where $f,p\in H^r\II$ and $q$ belongs to suitable Sobolev spaces to be specified below. The weak
formulation of problem \eqref{eqn:fde2} is given by: find $w\in V\equiv \Hd  1$ such that
\begin{equation}\label{eqn:var}
  \begin{aligned}
  A(w,\fy):=a(w,\fy)+b(w,\fy)=(f,\fy) \quad \forall\fy \in V,
  \end{aligned}
\end{equation}
where the bilinear forms $a(\cdot,\cdot)$ and $b(\cdot,\cdot)$ are defined on $V\times V$ by
\begin{equation}\label{eqn:bilin_form}
  \begin{aligned}
  a(\psi,\fy)=(\psi',\fy') \quad \text{and} \quad
  b(\psi,\fy)=({\DDR 0 {2-\alpha}} \psi,q\fy)+({\DDR 0 {2-\alpha}}\psi)(1) ( p,\fy).
  \end{aligned}
\end{equation}
First we show that $A(\cdot,\cdot)$ is bounded on $V\times V$. For $b(\cdot,\cdot)$, by Theorem \ref{thm:fracop}
we note that for $\psi\in V$
\begin{equation*}
   \| {\DDR 0 {2-\alpha}} \psi \|_{L^2\II}\le c\| \psi  \|_{\Hdi0{2-\al}}\le c\| \psi' \|_{L^2\II}.
\end{equation*}
By the identity $({\,_0\hspace{-0.3mm}I^{\al-1}_x} \psi)'=({\,_0\hspace{-0.3mm}I^{\al-1}_x} \psi')$ for
$\psi\in V$ \cite[Lemma 4.1]{JinLazarovPasciak:2013a} we have (with $\omega_{\alpha-1}(x)={(1-x)^{\alpha-2}}/{\Gamma(\alpha-1)}$)
\begin{equation*}
  \begin{aligned}
  |({\DDR 0 {2-\alpha}}\psi)(1)|&= |({\,_0\hspace{-0.3mm}I^{\al-1}_x} \psi') (1)|
  \le c\| \omega_{\alpha-1}\|_{L^2\II}  \| \psi'\|_{L^2\II}.
  \end{aligned}
\end{equation*}
Note that $\omega_{\alpha-1} \in L^2\II$ for $\al\in(3/2,2)$. Hence
\begin{equation}\label{eqn:boundb}
  \begin{aligned}
  |b(\psi,\fy)| &\le \| q \|_{L^\infty\II} \| {\DDR 0 {2-\alpha}} \psi \|_{L^2\II} \| \fy \|_{L^2\II}
  +  |({\DDR 0 {2-\alpha}}\psi)(1)| \| p \|_{L^2\II}  \| \fy \|_{L^2\II} \\
  &\le c  \| \psi \|_{V} \| \fy \|_{L^2\II}.
  \end{aligned}
\end{equation}

Now we turn to the well-posedness of the variational formulation \eqref{eqn:var}.
In case of $q \equiv p \equiv 0$, the bilinear form $A(\cdot,\cdot)$ is identical with $a(\cdot,\cdot)$
which recovers the standard Poisson equation and the well-posedness is well-known.
Next we consider the general case when $q $ and $p $ are not identically zero. To
this end, we make the following uniqueness
assumption on the bilinear form $A(\cdot,\cdot)$.
\begin{assumption} \label{ass:A}
Let the bilinear form $A(w,v)$ with $w,v\in V$ satisfy
\begin{itemize}
 \item[{$\mathrm{(a)}$}]  The problem of finding $w \in V$ such that $A(w,v)=0$ for all $v \in V$
           has only the trivial solution $w\equiv 0$.
 \item[{$(\mathrm{a}^\ast)$}] The problem of finding $v \in V$ such that $A(w,v)=0$ for all $w \in V$
    has only the trivial solution $v\equiv 0$.
\end{itemize}
\end{assumption}

Under Assumption \ref{ass:A}, the variational formulation \eqref{eqn:var} is stable.
\begin{theorem}\label{thm:well-posed}
Let Assumption \ref{ass:A} hold, $q\in L^\infty\II$ and
 $ p \in L^2\II$. Then for any $F\in V^*$, there
exists a unique solution $w\in V$ to
\begin{equation}\label{eqn:well-posed}
  A(w,\fy)= \langle F,\fy\rangle \quad \forall \fy\in V,
\end{equation}
where $\langle\cdot,\cdot\rangle$ denotes the duality between $V$ and its dual space $V^*=H^{-1}\II$.
\end{theorem}

\begin{proof}
The stability is proved by Petree-Tartar Lemma \cite[pp. 469, Lemma A.38]{ern-guermond}.
To this end, we define two operators $S \in \mathcal{L}(V;V^*)$ and $T \in \mathcal{L}(V;V^*)$ by
\begin{equation*}
  \langle Sw,\fy\rangle=A(w,\fy)\quad \text{and}\quad\langle Tw,\fy\rangle=-b(w,\fy),
\end{equation*}
respectively. Assumption \ref{ass:A}(a) shows the injectivity of the operator $S$. Further,
\begin{equation*}
  \begin{split}
   (Tw)(x)& = -\int_0^1 \frac{p(x) (1-y)^{\al-2}}{\Ga(\al-1)}w'(y)\,dy
   -\int_0^1 \frac{q(x)(x-y)^{\al-2}\chi_{(0,x)}(y)}{\Ga(\al-1)}w'(y)\,dy\\
    &=: (T_1 w)(x) + (T_2w)(x).
  \end{split}
\end{equation*}
We note that both $T_1$ and $T_2$ are compact from $V$ to $L^2\II$, since for $\alpha\in(3/2,2)$
both kernels are square integrable \cite[pp. 277, example 2]{Yoshida:1980}. Thus the operator
$T:V\to L^2\II$  is compact. By the definition of $a(\cdot,\cdot)$, we obtain
\begin{equation*}
    \|w\|_{V}^2 = a(w,w)= A(w,w)-b(w,w)
    \le c\left(\| Tu \|_{V^*}  +  \| Su  \|_{V^*}\right) \| w \|_V,
\end{equation*}
Now Petree-Tartar Lemma immediately implies that there exists a constant $c_0>0$
satisfying the following inf-sup condition
\begin{equation}\label{inf-suprl}
  c_0  \|u\|_V \le \sup_{v\in V } \frac {A(u,v)} {\|v\|_V}.
\end{equation}
This and Assumption \ref{ass:A}$(\hbox{a}^\ast)$ yield the existence  of
a unique solution $u\in V$ to \eqref{eqn:well-posed}.
\end{proof}

Now we state an improved regularity result for the case $\langle F,v\rangle
=(f,v)$, for some $f \in H^s\II$, $0\leq s\leq 1$.
\begin{theorem}\label{thm:reg-new}
Let Assumption \ref{ass:A} hold and $q\in L^\infty\II$ and $f\in L^2\II$. Then
the solution $w$ to problem \eqref{eqn:fde2} belongs to $\Hd 1 \cap H^2\II$
and satisfies
\begin{equation*}
  \| w \|_{H^2\II} \le c\|f\|_{L^2\II}.
\end{equation*}
Further, if $q,f\in H^1\II$, then it belongs to $H^{3}\II\cap \Hd 1$ and satisfies
\begin{equation*}
  \| w \|_{H^3\II} \le c \|f\|_{H^1\II}.
\end{equation*}
\end{theorem}
\begin{proof}
The existence and uniqueness of a solution $w \in V$ follows directly from Theorem \ref{thm:well-posed}.
Hence, it suffices to show the stability estimate. By Theorem \ref{thm:fracop}, $\DDR 0 {2-\alpha}w\in
H^{\al-1}\II$, and by Sobolev embedding theorem, $q\, {\DDR 0 {2-\alpha}w}\in H^{\alpha-1}\II$. Note that
problem \eqref{eqn:fde2} can be rewritten as
\begin{equation*}
-w''= \widetilde{f},
\end{equation*}
where $\widetilde{f}=-q\,\DDR 0 {2-\alpha}w -(\DDR 0 {2-\alpha}w)(1)p+f$.
The preceding discussion yields $\widetilde{f} \in L^2\II$ and $\|\widetilde{f}\|_{L^2\II}\leq c\|f\|_{L^2\II}$.
Hence, by standard elliptic regularity theory \cite{GilbargTrudinger:2001},
we deduce $u\in H^2\II \cap \Hd 1$. Further, if $ q,f\in H^1\II$, with this improved regularity on $w$, repeating
the preceding arguments gives $\widetilde{f}\in H^1\II$ and $\|\widetilde{f}\|_{H^1\II}\leq c\|f\|_{H^1\II}$, and
applying elliptic regularity theory again yields the desired estimate.
\end{proof}

The next result shows that Assumption \ref{ass:riem0} implies Assumption \ref{ass:A}(a).
\begin{lemma}
Let $p(x)= c_0x^{\mu-\al}-qx^\mu$ where $c_0$ 
is defined in \eqref{c-zero}. Then
Assumption \ref{ass:riem0} implies Assumption \ref{ass:A}$\mathrm{(a)}$.
\end{lemma}
\begin{proof}
Let $f=0$ in \eqref{varrl} and \eqref{eqn:var}. Suppose that $w \in V$ satisfies \eqref{eqn:var}.
Then by construction $u={\DDR0{2-\al}w}-(\DDR0{2-\al} w)(1)x^{\mu} \in \Hd{\al-1}$ and $(w',\fy') =
\langle -w'', \fy \rangle$ for $\fy\in C_0^\infty\II$ we have
\begin{equation*}
  \langle -\DDR0{\al}u + qu, \fy \rangle = 0 \quad \forall \fy\in C_0^\infty\II,
\end{equation*}
i.e., $-\DDR0{\al}u + qu= 0$ in the sense of distribution and in view of Theorem \ref{thm:reg-new}, also
in $L^2\II$. Now Assumption \ref{ass:riem0} yields $u=0$. Hence $w \in V$ satisfies
\begin{equation}\label{eqn:rec}
  \DDR0{2-\al}w=(\DDR0{2-\al}w)(1)x^{\mu}.
\end{equation}
by setting $\DDR0{2-\al}w = cx^{\mu}$,
the solution $w\in V$ of \eqref{eqn:rec}
is of the form $w(x)=c({\,_0\hspace{-0.3mm}I^{2-\al}_x}x^{\mu})(x)$. This together with the boundary condition
$w(1)=0$ yields $c=0$ and hence $w=0$.
\end{proof}

Once the solution $w$ to problem \eqref{eqn:fde2} is found, the solution to problem \eqref{eqn:fde} can be
found by the reconstruction formula \eqref{eqn:reconstruct_u}.
\begin{theorem}\label{thm:rec}
Let $f\in L^2\II$ and $q\in L^\infty\II$, and $w$ be the unique solution to \eqref{eqn:fde2}.
Then the representation $u$ given in \eqref{eqn:reconstruct_u} is a solution of problem \eqref{eqn:fde}.
\end{theorem}
\begin{proof}
For $f \in L^2\II$, by Theorem \ref{thm:reg-new}, there exists a unique solution $w\in \Hd 1\cap H^2\II$
to \eqref{eqn:fde2}. By Theorem \ref{thm:fracop}(a), we deduce
\begin{equation*}
w''= ({_0I_x^1}w') ''=({_0I_x^{2-\al}}({_0I_x^{\al-1}}w'))''
=({_0I_x^{2-\al}}({_0I_x^{\al-1}}w)')''={\DDR 0 {\alpha}}({\DDR 0 {2-\alpha}}w).
\end{equation*}
Upon substituting this into \eqref{eqn:fde2}, we get
\begin{equation*}
  - {\DDR 0 {\alpha}}({\DDR 0 {2-\alpha}}w) + q\,{\DDR 0 {2-\alpha}} w + ({\DDR 0 {2-\alpha}}w)(1)
     \left(c_0 x^{\mu-\al}-q(x)x^\mu\right) = f,
\end{equation*}
which together with the definition $ u= {\DDR 0 {2-\alpha}} w-({\DDR 0
{2-\alpha}}w)(1) x^{\mu}$ yields directly $-{\DDR 0 {\alpha}}u +qu=f$ in $L^2\II$. Clearly,
by the definition of $u$, $u(1)=0$, and further by Theorem \ref{thm:fracop} and the fact
that $w\in \Hd 1$, $ {\DDR 0 {2-\alpha} w} - ({\DDR 0 {2-\alpha}} w)(1)x^\mu \in
\widetilde{H}_L^{\alpha-1}\II$, and thus $u(0)=0$. Hence, $u$ is the solution to problem
\eqref{eqn:fde}.
\end{proof}

\subsection{Adjoint problem}\label{secc:adjprob}
To derive $L^2\II$ error estimates for the Galerkin approximation below,
we need the adjoint problem to \eqref{eqn:var}. For any $F \in V^*$, the adjoint
problem is to find $\psi \in V$ such that
\begin{equation}\label{eqn:dual}
    A(\fy,\psi)=\langle \fy,F \rangle \quad\forall \fy \in V.
\end{equation}
In the case of $\langle \fy,F \rangle = (\fy,f)$ for some $f \in L^2\II$, the strong form reads
\begin{equation}\label{eqn:dualstrong}
  \begin{aligned}
    -\psi'' + {\DDR1 {2-\al} }(q\psi) + \Gamma(\alpha-2)^{-1}(1-x)^{\al-3} (p,\psi)&= f \quad  \mbox{ in } D,\\
    \psi(0)=\psi(1)&=0.
  \end{aligned}
\end{equation}
We note that for $\alpha\in(3/2,2)$, the term $(1-x)^{\alpha-3}$ is not a function in $L^1\II$, and it
should be understood in the sense of distribution. In view of the identity $(1-x)^{\alpha-3}=-((1-x)^{
\alpha-2})'/(\alpha-2)$, and the fact that $(1-x)^{\alpha-2}$ belongs to the space $\widetilde{H}^{\alpha-2
+\beta}\II$, with $\beta\in(2-\alpha,1/2)$. Hence, $(1-x)^{\alpha-3}$ lies in the space $H^{\alpha-
3+\beta}\II \subset H^{-1}\II$.

\begin{theorem}\label{thm:reg-adj}
Let Assumption \ref{ass:A} hold, $q\in H^{1}\II$ and $f \in L^2\II$. Then there exists a
unique solution $\psi \in H^{\al-1/2}\II \cap \Hd 1$ to problem
\eqref{eqn:dual} and it satisfies for $\beta\in(2-\alpha,1/2)$
\begin{equation*}
  \| \psi \|_{H^{\al-1+\beta}\II} \le c \|f\|_{L^2\II}.
\end{equation*}
\end{theorem}
\begin{proof}
The unique existence of a solution $\psi\in V$ follows from Theorem \ref{thm:well-posed}.
To see the regularity, we rewrite the problem into
\begin{equation*}
    -\psi'' = - \DDR1 {2-\al} (q\psi) - \Gamma(\al-2)^{-1}(1-x)^{\al-3} (p,\psi) + f.
\end{equation*}
Under the given assumptions on the right hand side $f$ and the potential term $q$, and by the preceding discussions,
the right hand side belongs to $H^{\alpha-3+\beta}\II$. Thus by the standard elliptic regularity theory
\cite{GilbargTrudinger:2001}, the desired estimate follows.
\end{proof}

\begin{remark}\label{rmk:source-adj}
In Theorem \ref{thm:reg-adj}, the regularity assumption on the source term $f$ can be relaxed to $f\in H^{\al-3+\beta}\II$.
\end{remark}

Last we recall Green's function to the adjoint problem, i.e., for all $x\in D$
\begin{equation*}
  \begin{aligned}
    -G''(x,y) +  {_{y}^{ \kern -.1em R} \kern -.2em D^{2-\al}_{\kern -.1em 1}} (q G(x,y)) + \Gamma(\al-2)^{-1}(1-x)^{\al-3} (p,\psi)&= \delta_x(y)\quad  \mbox{ in } D,\\
    G(x,0)=G(x,1)&=0.
  \end{aligned}
\end{equation*}
By Sobolev embedding theorem, $\delta_x\in H^{-1+\beta}\II\subset H^{-1}\II$, $\beta\in(2-\alpha,1/2)$, and
thus the existence and uniqueness of $G(x,\cdot)\in \Hd 1$ follows directly from the stability
of the variational formulation. Moreover, by the argument in the proof of Theorem
\ref{thm:reg-adj} and Remark \ref{rmk:source-adj}, $G(x,\cdot)\in H^{\alpha-1+\beta}\II$.

\section{Galerkin finite element method}\label{sec:fem}
The variational formulation \eqref{eqn:var} enables us to develop a Galerkin FEM
for problem \eqref{eqn:fde}: first we approximate the solution $w$ to \eqref{eqn:fde2} by a Galerkin
finite element approximation $w_h$, and then reconstruct the solution to \eqref{eqn:fde} using
\eqref{eqn:reconstruct_u}, i.e.,
\begin{equation}\label{eqn:dicrecon}
  u_h={\DDR 0 {2-\alpha}} w_h-({\DDR 0 {2-\alpha}}w_h)(1) x^{\mu}.
\end{equation}
To this end, we divide the domain $D$ into quasi-uniform partitions with a maximum length $h$, and let
$V_h$ denote the resulting space of continuous piecewise polynomials of degree at most $k+1$, vanishing
at both end points of $D$. Thus, the functions in $V_h\subset \Hd1$ are piecewise linear if $k=0$, and
piecewise quadratic if $k = 1$. Since we consider only a right hand side $f\in L^2\II$ or $f\in H^1\II$,
we shall focus on the choice $k=0,1$ in our discussion. The space $V_h$ has the following approximation properties.
\begin{lemma}\label{fem-interp-U}
If $v \in H^\gamma\II \cap \Hd {1}$ with $ 1 \le \gamma \le 3$, then for $k=0,1$
\begin{equation}\label{approx-Uh}
\inf_{v_h\in V_h} \|v -v_h \|_{\Hd1} \le ch^{\min(\gamma-1,k+1)} \|v\|_{H^\gamma\II}.
\end{equation}
\end{lemma}

The Galerkin FEM is to find $w_h\in V_h$ such that
\begin{equation}\label{eqn:fem}
  A(w_h,v_h) = (f,v_h)\quad \forall v_h\in V_h.
\end{equation}
The computation of the stiffness matrix and mass matrix is given in Appendix \ref{app:stiff}.
We next analyze the stability of the discrete formulation \eqref{eqn:fem}, and derive (suboptimal) error
estimates for the approximations $w_h$ and $u_h$. First we have the following stability result. The proof
is identical with that in \cite[Lemma 5.2]{JinLazarovPasciak:2013a}, using a kick-back trick analogous
to Schatz \cite{Schatz-1974}. We sketch the proof for completeness.
\begin{theorem}\label{thm:disc-wellposed}
Let Assumption \ref{ass:A} hold, $f \in L^2\II$, and $q\in L^\infty\II$. Then there is
an $h_0$ such that for all $h\le h_0$ the finite element problem \eqref{eqn:fem}
has a unique solution $w_h\in V_h$, and further
\begin{equation}\label{eq:stability}
  \|w_h\|_{H^1\II}\leq c\|f\|_{L^2\II}.
\end{equation}
\end{theorem}
\begin{proof}
We first define the Ritz projection $R_h: V \to V_h$ by $((R_h \fy)',\psi')=(\fy',\psi') $
for all $\psi \in V_h$. Then for $v_h \in V_h \subset V$ we have
$$
c_0 \| v_h' \|_{L^2\II}\le \sup_{\fy \in V} \frac{A(v_h,\fy)}{\| \fy' \|_{L^2\II}}
\leq  \sup_{\fy \in V} \frac{A(v_h,\fy-R_h\fy)}{\| \fy' \|_{L^2\II}}
+  \sup_{\fy \in V} \frac{A(v_h,R_h\fy)}{\| \fy' \|_{L^2\II}}=:I+II.
$$
Then by \eqref{eqn:boundb} and Theorem \ref{thm:reg-new} we have
$$
I= \sup_{\fy \in V} \frac{b(v_h,\fy-R_h\fy)}{\| \fy' \|_{L^2\II}}
\le  c\sup_{\fy \in V} \frac{\| v_h' \|_{L^2\II}  \| \fy-R_h\fy \|_{L^2\II}}{\| \fy' \|_{L^2\II}}\le c_1h\| v_h' \|_{L^2\II}.
$$
Further the second term $II$ could be bounded as follows by using
the inequality $\| (R_h \fy)' \|_{L^2\II} \le \| \fy' \|_{L^2\II}$ and
the fact that $R_h \fy \in V_h$
$$II \le \sup_{\fy \in V} \frac{A(v_h,R_h\fy)}{\| (R_h\fy)' \|_{L^2\II}}
\le \sup_{\fy_h \in V_h} \frac{A(v_h,\fy_h)}{\| \fy_h' \|_{L^2\II}}.$$
Now by choosing $h_0=c_0/(2c_1)$ we derive the following {\it inf-sup} condition:
\begin{equation}\label{inf-sup-disc}
    \|  v_h \|_V \le c \sup_{\fy_h\in V_h}  \frac{A(v_h,\fy_h)}{\| \fy_h\|_V}.
\end{equation}
This shows that the corresponding stiffness matrix is nonsingular and the
existence of a unique discrete solution $u_h\in V_h$ follows.
The estimate \eqref{eq:stability} is a direct consequence of \eqref{inf-sup-disc} and this completes the proof.
\end{proof}

Now we turn to the error analysis, and focus on the case $f\in H^1\II$.
\begin{theorem}\label{thm:error-w}
Let Assumption \ref{ass:A} hold, and $f,q\in H^1\II $. For the FEM of piecewise $(k+1)$'s degree
polynomials (k=0,1), there is an $h_0$ such that for all $h\le h_0$, the solution $w_h$ to
problem \eqref{eqn:fem} satisfies with $\beta\in (2-\alpha,1/2)$
\begin{equation*}
   \|w-w_h\|_{L^2\II} + h^{\al-2+\beta}\|(w-w_h)'\|_{L^2\II}
   \le C h^{\al+k-1+\beta} \|f\|_{H^1\II}.
\end{equation*}
\end{theorem}

\begin{proof}
The error estimate in the $\Hd{1}$-norm follows directly from C\'{e}a's lemma,
\eqref{inf-sup-disc} and the Galerkin orthogonality. Specifically, for all $h\le h_0$ and
any $ \chi \in V_h$ we have
\begin{equation*}
   \| w_h-\chi \|_{V}
    \le c\sup_{v_h \in V_h} \frac{A(w_h-\chi,v_h)}{ \|v_h\|_V}
    \le c\sup_{v_h\in V_h} \frac{A(w-\chi,v_h)}{ \|v_h\|_V}
    \le c\| w-\chi \|_{V}.
\end{equation*}
Then the desired $\Hd1$-estimate follows from Lemma the triangle inequality and \ref{fem-interp-U} by
\begin{equation*}
\begin{split}
    \| w-w_h \|_V &\le c\inf_{\chi\in V_h} \| w-\chi \|_{V} \le ch^{k+1}\| f \|_{H^1\II}.
\end{split}
\end{equation*}
Then we apply Nitsche's trick to establish the $L^2\II$-error estimate.
To this end, we consider the adjoint problem \eqref{eqn:dual} with $f=w-w_h$, i.e.
\begin{equation*}
\begin{split}
    \| w-w_h \|_{L^2\II}^2 = A(w-w_h, \psi)= A(w-w_h,\psi-\psi_h),
\end{split}
\end{equation*}
for any $w_h\in V_h$. Then Lemma \ref{fem-interp-U} and Theorem \ref{thm:reg-adj} yield
for any $\beta\in[1-\alpha/2,1/2)$
\begin{equation*}
\begin{split}
    \| w-w_h \|_{L^2\II}^2 &\le c\|w-w_h\|_{V}\inf_{\psi_h\in V_h}\|\psi-\psi_h\|_{V}\\
    &\le ch^{\al+k-1+\beta} \| f \|_{H^1\II} \| w-w_h \|_{L^2\II}.
\end{split}
\end{equation*}
This completes the proof of the theorem.
\end{proof}

Below we analyze the convergence of the approximation $u_h$, reconstructed from $w_h$ using \eqref{eqn:dicrecon}. We divide the convergence
analysis into several lemmas. First we estimate the leading term ${\DDR 0 {2-\alpha}}w_h(x)$.
\begin{lemma}\label{thm:error-I}
Let the assumptions in Theorem \ref{thm:error-w} hold, and $w$ and $w_h$ be solutions of
\eqref{eqn:var} and \eqref{eqn:fem}, respectively. Then for $e=w-w_h$, there holds
with $\beta\in (2-\alpha,1/2)$
\begin{equation*}
   \|\DDR 0 {2-\alpha}e\|_{L^2\II} \le ch^{\al+k-1+\beta} \|f\|_{H^1\II}.
\end{equation*}
\end{lemma}
\begin{proof}
Recall that $\al\in(3/2,2)$, $2-\alpha\in(0,1/2)$, and thus the spaces $\Hd{2-\alpha}$ and $H^{2-\alpha}\II $
are equal, and further $\|\DDR 0 {2-\alpha}\cdot\|_{L^2\II}$ induces an equivalent norm on $H^{2-\alpha}\II$
\cite{LiXu:2009}. By a standard duality argument, we deduce
\begin{equation*}
\begin{split}
  \|\DDR 0 {2-\alpha}e\|_{L^2\II} &\le c\|e\|_{H^{2-\al}\II}
   = c\sup_{\fy\in H^{-2+\al}\II}\frac{\langle e,\fy \rangle}{\| \fy \|_{H^{-2+\al}\II}}\\
    &= c\sup_{\fy\in H^{-2+\al}\II}\frac{A(e,g_\fy )}{\| \fy \|_{H^{-2+\al}\II}},
\end{split}
\end{equation*}
where $g_\fy$ is the solution to the adjoint problem $\langle v,\phi \rangle=A(v,g_\fy)$, for
all $v\in V$. By Theorem \ref{thm:reg-adj}, $g_\fy \in H^{\al-1+\beta}\II$. Let $\Pi \fy\in V_h$
 be the standard Lagrange finite element interpolant of $\fy$. Then by
Galerkin orthogonality and the continuity of the bilinear form
\begin{equation*}
\begin{split}
  A(e,g_\fy )&= A(e,g_\fy-\Pi g_\fy )\le c\| e' \|_{L^2\II}  \| (g_\fy-\Pi g_\fy)'  \|_{L^2\II}\\
  &\le ch^{\al+k-1+\beta} \|f\|_{H^1\II} \| g_\fy \|_{H^{\al-1+\beta}\II}\\
  &\le ch^{\al+k-1+\beta} \|f\|_{H^1\II} \| \fy \|_{H^{-2+\al}\II}.
\end{split}
\end{equation*}
\end{proof}

Next we provide an $L^\infty\II$ estimate on the term $e=w-w_h$.
\begin{lemma}\label{thm:error-Linf}
Let the assumptions in Theorem \ref{thm:error-w} hold, and $w$ and $w_h$ be solutions of \eqref{eqn:var}
and \eqref{eqn:fem}, respectively. Then for $e=w-w_h$ and $\beta\in (2-\alpha,1/2)$, there holds
\begin{equation*}
   \|e\|_{L^\infty\II} \le c h^{\al+k-1+\beta} \|f\|_{H^1\II}.
\end{equation*}
\end{lemma}
\begin{proof}
Using the weak formulation of $G(x,y)$
and Galerkin orthogonality, we have for any $\fy_h\in V_h$
\begin{equation*}
   e(x)= A(e, G(x,\cdot))=A(e, G(x,\cdot)-\fy_h).
\end{equation*}
Then by Theorem \ref{thm:error-w}, we obtain for any $\beta\in(2-\alpha,1/2)$
\begin{equation*}
|e(x)| \le c\|e\|_{H^1\II}\inf_{\fy_h\in V_h} \|G(x,\cdot)-\fy_h\|_{H^1\II}
  \le ch^{\alpha+k-1+\beta}\| f \|_{H^1\II},
\end{equation*}
where the last inequality follows from $G(x,\cdot)\in H^{\alpha-1+\beta}\II \subset H^1(D)$ and Lemma \ref{fem-interp-U}.
\end{proof}

The next result gives an estimate on the crucial term $|(\DDR 0 {2-\alpha}e)(1)|$.
\begin{lemma}\label{thm:error-II}
Let the assumptions in Theorem \ref{thm:error-w} hold, and $w$ and $w_h$ be solutions of
\eqref{eqn:var} and \eqref{eqn:fem}, respectively. Then for $e=w-w_h$, there holds with $\beta\in(2-\alpha,1/2)$
\begin{equation*}
   |(\DDR 0 {2-\alpha}e)(1)| \le ch^{\al+k-1+\beta} \|f\|_{H^1\II}.
\end{equation*}
\end{lemma}
\begin{proof}
By the Galerkin orthogonality, we have
\begin{equation*}
  (e ',\fy_h') + (\ {\DDR 0 {2-\alpha}}e,q\fy_h) + ({\DDR 0 {2-\alpha}}e)(1)
     (p,\fy_h)=0 \quad \forall\fy_h\in V_h.
\end{equation*}
Note that $p(x)=\frac{\Gamma(\mu+1)}{\Gamma(1+\mu-\al)} x^{\mu-\al}-q(x)x^\mu$ is smooth for large
$\mu$. Without loss of generality, we may assume that $x=1/2$ is a grid point and let $\fy_h=x\chi_{[0,1/2)}
+(1-x)\chi_{(1/2,1]}\in V_h$ with $|(p,\fy_h)|:=c_1> 0$. Then  we obtain
\begin{equation*}
        c_1|({\DDR 0 {2-\alpha}}e)(1)| \le |(e',\fy_h')| + |(\ {\DDR 0 {2-\alpha}}e,q\fy_h)|=:I+II.
\end{equation*}
It suffices to bound the terms on the right hand side. The second term $II$ can be bounded using Lemma \ref{thm:error-I} as
\begin{equation*}
         II \le \| \ {\DDR 0 {2-\alpha}}e \|_{L^2\II}\| \fy_h \|_{L^2\II} \| q\|_{L^\infty\II}
           \le c\| \ {\DDR 0 {2-\alpha}}e \|_{L^2\II} \le ch^{\al+k-1+\beta} \|f\|_{H^1\II}.
\end{equation*}
and the first term $I$ can be bounded by Lemma \ref{thm:error-Linf} by
\begin{equation*}
        I\le  | \int_0^{1/2} e'(x) dx -\int_{1/2}^1 e'(x) dx | = 2|e(1/2)|
        \le ch^{\al+k-1+\beta} \| f \|_{H^{1}\II}.
\end{equation*}
This completes the proof of the lemma.
\end{proof}

Now by the triangle inequality, we arrive at the following $L^2\II$ estimate for the approximation $u_h$.
\begin{theorem}\label{thm:error-u}
Let Assumption \ref{ass:A} hold, $f,q\in H^1\II $.
Then there is an $h_0$ such that for all $h\le h_0$,
the solution $u_h$ satisfies that for any $\beta\in(2-\alpha,1/2)$
\begin{equation}\label{eqn:error-u}
   \|u-u_h\|_{L^2\II} \le ch^{\al+k-1+\beta} \| f \|_{H^{1}\II}.
\end{equation}
\end{theorem}

\begin{remark}\label{rem:error-mu2}
By Remark \ref{rem:reg-mu2}, we may choose $\mu=\al-1$, for which the error estimate follows similarly.
The only difference is the bound on $|(\DDR 0 {2-\alpha}e)(1)|$ in case of $q=0$. By the definition of
$(\DDR 0 {2-\alpha}e)(1)$, we have
\begin{equation*}
\begin{split}
|{\DDR 0 {2-\alpha}}e(1)|&= \frac{1}{\Gamma(\al-1)}\left|\int_0^1 (1-x)^{\al-2} e'(x)dx\right|
=\frac{1}{\Gamma(\al)}\left|\int_0^1 ((1-x)^{\al-1})' e'(x)\,dx\right|\\
&\le \frac{1}{\Gamma(\al)}\left|\int_0^1 ((1-x)^{\al-1}-(1-x))' e'(x)\,dx\right|+\frac{1}{\Gamma(\al)}\left|\int_0^1e'(x)\,dx\right|.\\
\end{split}
\end{equation*}
The second term vanishes due to $e(0)=e(1)=0$. Hence it suffices to establish an estimate
on first term. Since the transformed problem reproduces Poisson's equation, by the Galerkin
orthogonality $(e',\fy_h')=0$ and the fact that $\fy=(1-x)^{\al-1}-(1-x)\in \Hd 1\cap H^{\al-1+\beta}
\II$ with $\beta\in(2-\alpha,1/2)$ , we have by Lemma \ref{fem-interp-U}
\begin{equation*}
\begin{split}
\left|(\fy', e'(x))\right| \le c \| e' \|_{L^2\II} \inf_{\fy_h\in V_h} \| \fy'-\fy_h'  \|_{L^2\II}
&\le ch^{\al+k-1+\beta} \| f \|_{H^1\II}.
\end{split}
\end{equation*}
Thus the $L^2\II$ estimate \eqref{eqn:error-u} holds also for the choice $\mu=\al-1$.
\end{remark}

Next, we derive an optimal $L^2\II$ error estimate for all $\al\in(1,2)$ provided that $q=0$, $\mu=\al-1$ and $f$ is smooth enough.
\begin{theorem}\label{thm:L2-optimal}
Assume $q=0$ and $\mu=\al-1$. Then for all $\al\in(1,2)$ there  holds
\begin{equation*}
   \|u-u_h\|_{L^2\II} \le ch^{\al+k} \| f \|_{W^{1,\infty}\II}.
\end{equation*}
\end{theorem}
\begin{proof}
For $q=0$ and $\mu=\alpha-1$, the transformed problem is the standard  one-dimensional Poisson's equation
$$ -w''=f \quad \mbox{in } D,\quad\text{with}\quad w(0)=w(1)=0. $$
Then the solution $w_h$ of the discrete problem \eqref{eqn:fem} satisfies \cite{Wheeler:1973,DouglasDupont:1974}
$$ \| w-w_h \|_{W^{s,\infty}\II}+ \|w-w_h\|_{W^{s,2}\II} \le c h^{k+2-s} \| f \|_{W^{1,\infty}\II},\quad s=0,1.  $$
Now let $e=w-w_h$ and we have by interpolation
\begin{equation}\label{eqn:es-I}
 \|\DDR 0 {2-\alpha}e\|_{L^2\II} \le \| e \|_{H^{2-\al}\II}\le ch^{\al+k} \| f \|_{W^{1,\infty}\II}.
\end{equation}
Hence it suffices to bound $|(\DDR 0 {2-\alpha}e)(1)|$. Since $e\in \Hd1$, we have for $\delta\in(0,1)$
\begin{equation*}
\begin{split}
|(\DDR 0 {2-\alpha}e)(1)| &= \frac{1}{\Gamma(\al-1)}\bigg| \int_0^1 (1-s)^{\al-2} e'(s)\bigg|\\
&\le c\left(\bigg| \int_0^{1-\delta} (1-s)^{\al-2} e'(s)\,ds\bigg|+\bigg| \int_{1-\delta}^1 (1-s)^{\al-2} e'(s)\,ds\bigg|\right).
\end{split}
\end{equation*}
Then the second term can be easily bounded by
\begin{equation*}
 \bigg| \int_{1-\delta}^1 (1-s)^{\al-2} e'(s)\,ds \bigg| \le \int_{1-\delta}^1 (1-s)^{\al-2} \,ds \|  e' \|_{L^\infty\II}
\le c\delta^{\al-1} h^{k+1}  \| f \|_{W^{1,\infty}\II},
\end{equation*}
while the first term can be bounded using integration by parts
\begin{equation*}
    \begin{split}
      \bigg| \int_0^{1-\delta} (1-s)^{\al-2} e'(s)\,ds\bigg|
       &\le c\left(\bigg| (1-s)^{\al-2}e(s)\big|_0^{1-\delta}\bigg|+\bigg|\int_0^{1-\delta} \frac{(1-s)^{\al-1}}{\al-1} e(s)\,ds\bigg| \right)\\
        &\le c(\delta^{\al-2}+1-\delta^\al)h^{k+2} \| f \|_{W^{1,\infty}\II}.
    \end{split}
\end{equation*}
Now choosing $\delta=h$ yields the following estimate
\begin{equation*}
|(\DDR 0 {2-\alpha}e)(1)| \le ch^{\al+k}.
\end{equation*}
This together with \eqref{eqn:es-I} gives an optimal $L^2\II$-error estimate
\begin{equation*}
 \| u-u_h  \|_{L^2\II}\le \| \DDR 0 {2-\alpha}e  \|_{L^2\II}+ c|(\DDR 0 {2-\alpha}e)(1)|\le ch^{\al+k}  \| f \|_{W^{1,\infty}\II}.
\end{equation*}
\end{proof}

\section{Eigenvalue problem}\label{sec:eig}
Now we apply the new approach to the following fractional Sturm-Liouville
problem (FSLP): find $u$ and $\lambda\in \mathbb{C}$ such that
\begin{equation}\label{eqn:eig}
  \begin{aligned}
   -\DDR 0 \alpha u + qu &= \lambda u\quad \mbox{ in } D,\\
     u(0)=u(1) &= 0.
  \end{aligned}
\end{equation}
The eigenvalue problem is important in studying the dynamics of superdiffusion processes.
However, the accurate computation of the eigenvalues and eigenfunctions is challenging,
due to the presence of a singularity in the eigenfunction. In \cite{JinLazarovPasciak:2013a},
a finite element method with piecewise linear finite elements was developed for the problem.
Numerically, a second-order convergence of the eigenvalue approximations is observed, but
the theoretical convergence rate of eigenfunction approximations is of order $O(h^{\al-1})$ in the $L^2\II$ norm which is
very slow. In this part, we develop an efficient method for problem \eqref{eqn:eig} by
extending the new approach in Sections \ref{sec:soltheory} and \ref{sec:fem}.

Proceeding like in section \ref{sec:soltheory}, we deduce that the weak formulation of the
Sturm-Liouville problem reads: find $w\in V$ and $\la\in \mathbb{C}$ such that
\begin{equation}\label{eqn:FSLP_new}
    A(w,\fy)=\la (\DDR 0{2-\al}w-(\DDR 0{ 2- \al}w)(1)x^\mu,\fy) \quad \forall \fy \in V.
\end{equation}
Then we define $u$ by
\begin{equation*}
   u = {\DDR 0{2-\al}} w-(\DDR 0{2-\alpha}w)(1)x^\mu.
\end{equation*}
Then $\lambda$ is the eigenvalue and $u$ is the corresponding eigenfunction.
Accordingly, the discrete problem is given by: find $w_h\in V_h$ and $\la_h\in \mathbb{C}$ such that
\begin{equation}\label{eqn:FSLP_new_disc}
\begin{split}
    A(w_h,\fy)&=\la_h (\DDR 0{\al-2}w_h-(\DDR 0{\al-2}w_h)(1)x^\mu,\fy) \quad \forall \fy \in V_h,\\
    u_h&={\DDR 0{\al-2}w_h}-(\DDR 0{\al-2}w_h)(1)x^\mu.
\end{split}
\end{equation}
and $\{\la_h,w_h\}$ is an approximated eigenpair of the transformed FSLP \eqref{eqn:FSLP_new}.

We shall follow the notation and use some fundamental results from
\cite{Osborn:1975,Babuska-Osborn}. To this end, we introduce the operator $T: L^2\II\to \Hd 1$ defined by
\begin{equation}\label{T-operator}
T f \in \Hd 1, \quad A(Tf,\fy) = (f,\fy) \quad \forall \fy \in V.
\end{equation}
Obviously, $T$ is the solution operator of the source problem \eqref{eqn:fde2}. By Theorem
\ref{thm:reg-new}, the solution operator $T$ satisfies the following smoothing property:
\begin{equation*}
\| Tf \|_{H^2\II} \le c\|f\|_{L^2\II}.
\end{equation*}
Since $H^2\II$ is compactly embedded into $H^1\II$ \cite{AdamsFournier:2003}, we deduce
that the operator $T:~ L^2\II\to \Hd 1$ is compact. Next we define an operator $S:
\Hd 1 \to L^2\II$ by
\begin{equation}\label{eqn:Op-S}
    Sw={\DDR 0{2-\al}w}-(\DDR 0{2-\al}w)(1)x^\mu.
\end{equation}
\begin{lemma}\label{lem:opS}
The operator $S: \Hd 1\to L^2\II $ defined in \eqref{eqn:Op-S} is compact.
\end{lemma}
\begin{proof}
We observe that for $w\in \Hd 1$
\begin{equation*}
  \|Sw\|_{L^2\II} \leq \|\DDR 0 {2-\alpha} w\|_{L^2\II} + |(\DDR 0 {2-\al} w)(1)|\|x^\mu\|_{L^2\II}.
\end{equation*}
By Theorem \ref{thm:fracop}, we have
\begin{equation*}
  \|\DDR 0 {2-\alpha} w\|_{L^2\II}\leq c\|w\|_{H^{2-\alpha}\II}.
\end{equation*}
Meanwhile, by Sobolev embedding theorem \cite{AdamsFournier:2003} and norm equivalence on the
space $\Hd s$ \cite{JinLazarovPasciak:2013a}, there holds for $\alpha-1>s>1/2$, i.e., $1/2<s+2-\alpha<1$,
\begin{equation*}
  \begin{aligned}
    |(\DDR 0 {2-\al} w)(1)| & \leq c\|\DDR 0 {2-\alpha} w\|_{H^s\II}\leq c\|\DDR 0 s (\DDR 0 {2-\alpha} w)\|_{L^2\II}\\
     & =c\|\DDR 0 {s+2-\alpha} w\|_{L^2\II}\leq  c\|w\|_{H^{s+2-\alpha}\II}.
  \end{aligned}
\end{equation*}
These two estimates implies that the operator is bounded from $\Hd {s+2-\alpha}$ to $L^2\II$,
which together the compactness of the embedding from $\Hd 1$ into $\Hd {s+2-\alpha}$ yields the desired
compactness.
\end{proof}

Then the FSLP \eqref{eqn:FSLP_new} can be rewritten as to find
$ w \in V, \ \ \mbox{such that} \, \, A(w, \fy) = \lambda (Sw, \fy)$
$ \forall \fy \in V $ or equivalently $ TSw=\lambda^{-1} w$.
Now after applying the operator $S$ to this equality and noting that $Sw=u \in L^2\II$ we
get the problem in oprator form:
find $(\la, u)\in \mathbb{C}
\times L^2\II$ such that
\begin{equation*}
 \la^{-1} u = ST u,
\end{equation*}
i.e., $(\la^{-1}, u)$ is an eigenpair of the operator $ST$. By Lemma \ref{lem:opS}, the operator $S:\Hd 1\to
L^2\II$ is bounded and compact, and thus $ST:~ L^2\II\to L^2\II$ is a compact operator. With the help of this
correspondence, the properties of the eigenvalue problem \eqref{eqn:eig} can be derived from the
spectral theory for compact operators \cite{Yoshida:1980,DunfordSchwartz:1988}. Let $\sigma(ST)
\subset \mathbb{C}$ be the set of all eigenvalues of $ST$ (or its spectrum), which is known to be a countable
set with no nonzero limit points. By Assumption \ref{ass:A} on the bilinear form $a(u; v)$, zero
is not an eigenvalue of $ST$. Furthermore, for any $\mu\in\sigma(ST)$, the space $N(\mu I-ST)$,
where $N$ denotes the null space, of eigenvectors corresponding to $\mu$ is finite dimensional.

Now let $T_h: ~V_h \to V_h$ be a family of operators for $0<h<1$ defined by
\begin{equation}\label{Th-operator}
  T_h f \in V_h, \quad A(T_hf,\fy) = (f,\fy)\quad \forall\fy \in V_h. 
\end{equation}
Then the discrete FSLP \eqref{eqn:FSLP_new_disc} can be written as: to find
$ w_h \in V_h, \ \ \mbox{such that} \, \, A(w_h, \fy) = \lambda_h (Sw_h, \fy)$
$ \forall \fy \in V $ or equivalently $ T_hSw_h=\lambda_h^{-1} w_h$, with $u_h=Sw_h$.
Hence the discrete problem in operator form reads: to find
$(\la_h, u_h)\in \mathbb{C}\times L^2\II$ such that
$$ \la_h^{-1} u = ST_h u.$$
By Theorem \ref{thm:error-u}, the operator $ST_h$ converges to $ST$ in $L^2\II$.
Further, the operator sequence $\{ST_h\}_{h>0}$ is collectively compact on $L^2\II$, i.e.,
the set $\{ST_hf: \|f\|_{L^2\II}\leq 1\}$ is compact in $L^2\II$. To see this, we note that
by the  discrete inf-sup condition, $\|T_hf\|_{H^1\II}\leq c$, cf. Theorem \ref{thm:disc-wellposed},
and thus the set $\{T_hf:
\|f\|_{L^2\II}\leq 1\}$ is uniformly bounded in $\Hd 1$, and the claim follows from the
compactness of the operator $S:\Hd 1\to L^2\II$ from Lemma \ref{lem:opS}. Hence, we can apply
the approximation theory \cite{Osborn:1975} of compact operators.
Specifically, let $\mu=\lambda^{-1}\in \sigma(ST)$ be an eigenvalue of $ST$ with algebraic multiplicity
$m$. Then $m$ eigenvalues of $ST_h$, $\mu_h^j$, $j=1,2,\ldots,m$,
of $ST_h$ will converge to $\mu$, where the eigenvalues $\mu_h^j$ are counted according to the
algebraic multiplicity of $\mu_h^j$ as eigenvalues of $ST_h$.

Now we state the main result for the spectral approximation.
It follows directly from \cite[Theorems 5 and 6]{Osborn:1975} and Theorem \ref{thm:error-u}.
\begin{theorem}\label{thm:eig-approx}
Let Assumption \ref{ass:A} hold and $q\in H^1\II$. For $\la^{-1} \in \sigma(ST)$, let $\delta$ be its
ascent, i.e., the smallest integer $m$ such that $N((\la^{-1}-ST)^m) =
N((\lambda^{-1} -ST)^{m+1})$.
\begin{itemize}
  \item[(i)] For any $\gamma<\alpha+k-1/2$, there holds
   \begin{equation*}
     |\lambda - \lambda_h^j| \le C h^{\gamma/\delta}.
   \end{equation*}
   \item[(ii)]
   Let $\la_h^{-1}$ be an eigenvalue of $ST_h$ such that $\lim_{h\to 0} \la_h=\la$ with $\la \in \sigma(ST)$.
   Suppose for each $h$, $u_h$ is a unit vector satisfying $((\lambda_h^j)^{-1} - ST_h)^k u_h=0$ for some
   positive integer $k\le \delta$. Then, for any integer $l$ with $k\le l \le \al$, there is a vector $u$
   such that $(\la^{-1}-ST)^lu=0$ and for any $\gamma<\alpha+k-1/2$,
   \begin{equation*}
     \| u-u_h  \|_{L^2\II} \le Ch^{\gamma/\delta}.
   \end{equation*}
\end{itemize}
\end{theorem}

\begin{remark}
It is known that in case of $q=0$, all eigenvalues to \eqref{eqn:eig} are simple  \cite[Section 4.4]{Sedletskii:2000},
i.e., $\delta=1$ in Theorem \ref{thm:eig-approx}. Numerically we observe that the
eigenvalues to \eqref{eqn:eig} are always simple. When using piecewise linear finite
elements, the convergence rate of the new approach in Theorem
\ref{thm:eig-approx} is better than that for the standard Galerkin method, which
has a convergence rate $Ch^{\gamma/\delta}$, for any $\gamma<\alpha-1$
\cite[Theorem 6.1]{JinLazarovPasciak:2013a}. This shows the advantage of the new approach.
\end{remark}

\section{Numerical results and discussions}\label{sec:numer}
In this section, we present numerical results to illustrate the efficiency and accuracy of the
new approach and to verify our theoretical findings. We shall discuss the source problem and
the Sturm-Liouville problem separately.
\subsection{Source problem}
For the source problem \eqref{eqn:fde}, we consider the following three different right hand sides:
\begin{itemize}
  \item[(a)] The source term $f(x)=x(1-x)$ belongs to $\Hd {1+\beta}$ for any $\beta\in [0,1/2)$.
  \item[(b)] The source term (b1) $f(x)=1$ and  (b2) $ f(x)=(1-x)^{\frac35}$ belong to
   the space $H^1\II\cap \Hd{\beta}$ for any $\beta\in [0,1/2)$.
  \item[(c)] The source term $f(x)=\chi_{[0,1/2]}$ belongs to $\Hd{\beta}$ for any $\beta\in [0,1/2)$.
\end{itemize}
The computations were performed on a uniform mesh with a mesh size $h=1/2^m$, $m\in \mathbb{N}$.
We note that if the potential $q$ is zero, the exact solution $u$ can be computed explicitly.
For the case $q\neq 0$, the exact solutions are not available in closed form, and hence we
compute the reference solution on a very refined mesh with a mesh size $h= 1/2^{12}$. For each example,
we consider three different $\alpha$ values, i.e., $1.55$, $1.75$ and $1.95$, and present the
$L^2\II$-norm of the error $e=u-u_h$.

\subsubsection{Numerical results for example (a)}
For this very smooth source, we consider the simple case $q=0$. The exact solution $u(x)$ is given by
$u(x) =  \tfrac{1}{\Gamma(\alpha+2)}(x^{\alpha-1}-x^{\alpha+1})-\tfrac{2}{\Gamma(\alpha+3)}(x^{\alpha-1}
-x^{\alpha+2})$, and it belongs to $\Hdi0{\alpha-1+\beta}$ with $\beta\in(2-\alpha,1/2)$ due to the
presence of the term $x^{\al-1}$, despite the smoothness of the right hand side $f$. Thus the standard
Galerkin FEM converges slowly; see \cite[Table 1]{JinLazarovPasciak:2013a}. Numerical results for the
new approach are presented in Table \ref{tab:source-a-q0}. In the table, $P1$ and $P2$ denote piecewise
linear and piecewise quadratic FEMs, respectively. \texttt{rate} refers to the empirical convergence rate,
and the numbers in the bracket denote theoretical rates. The numerical results show $O(h^{\al})$ and
$O(h^{\al+1})$ convergence for $P1$ and $P2$ FEMs, respectively. Hence, the $L^2\II$-error estimate in Theorem \ref{thm:error-u} is
suboptimal: the empirical ones are one half order higher than the theoretical one. The suboptimality is
attributed to the low regularity of the adjoint problem \eqref{eqn:dual}, used in Nitsche's trick.
Although not presented, we note that with the choice $\mu=\al-1$, the optimal convergence rate in
Theorem \ref{thm:L2-optimal} can be fully confirmed.

\begin{table}[hbt!]
\caption{The $L^2\II$-norm of the error for example (a) with $q=0$, $\mu=4$, $\al=1.55, 1.75, 1.95$, $h=1/2^m$.}
\label{tab:source-a-q0}
\vspace{-.3cm}
\begin{center}
     \begin{tabular}{|c|c|cccccc|c|}
     \hline
     $\al$  & $m$ &$3$ &$4$ &$5$ & $6$ & $7$ & $8$  &rate \\
     \hline
     $1.55$ & $P1$ &2.62e-3 &9.28e-4 &3.20e-4 &1.09e-4 &3.68e-5 &1.22e-5 & 1.55 (1.05) \\

     & $P2$   &2.30e-5 &3.96e-6 &6.79e-7 &1.16e-7 &1.98e-8 & 3.39e-9 & 2.55 (2.05)\\
    \hline
     $1.75$ & $P1$    &7.89e-4 &2.26e-4 &6.47e-5 &1.86e-5 &5.34e-6 & 1.53e-6 & 1.80 (1.25)\\

     & $P2$    & 1.11e-5 &1.69e-6 &2.54e-7 &3.80e-8 &5.66e-9 &8.39e-10 & 2.74 (2.25)\\
     \hline
     $1.95$ & $P1$     &3.06e-4 & 7.74e-5	&1.95e-5 &4.93e-6 &1.24e-6 &3.11e-7& 1.98 (1.45)\\

     & $P2$   &5.38e-6	&7.03e-7 &9.15e-8 &1.18e-8 &1.53e-9 &1.98e-10 & 2.95 (2.45)\\
     \hline
     \end{tabular}
\end{center}
\end{table}

\subsubsection{Numerical results for example (b)}
In Table \ref{tab:source-b-q}, we present numerical results for example (b1) with $q(x)=x$. Since both the
source term $f$ and the potential $q$ belong to $H^1\II$, by Theorem \ref{thm:reg-new}, $w$ belongs to
$H^3\II\cap \Hd 1$, and the $L^2\II$-error achieves a rate $O(h^{\al+k-1/2})$ for $k=0,1$. The empirical
$L^2\II$ rate is one half order higher than the theoretical one. Next we compare the new approach with the
singularity enhanced FEM developed in \cite{JinZhou:2014}. Since the regular part $u_r$ (i.e., the part
of the solution $u$ apart from the leading singularity $x^{\alpha-1}$) only
belongs to $H^{\al+\beta}\II$ due to $f,q\in \Hd{\beta}$, even with the $P2$ FEM, the approach in
\cite{JinZhou:2014} can only achieve a convergence rate slower than that in Theorem \ref{thm:error-u},
and the new approach requires less regularity on the potential $q$ and source term $f$.
In Table \ref{tab:source-small_al}, we show numerical results for $\al<1.5$, which is not covered by our
theory. Interestingly, the numerical results indicate that
our scheme converges equally well with the order $O(h^{\al+k})$ in this case.

Further numerical results for different $\mu$ values are presented in Table \ref{tab:source-b-q-mu}.
By Remarks \ref{rem:reg-mu2} and \ref{rem:error-mu2}, the choice $\mu=\al-1$
achieves the rate $O(h^{\al+k-1+\beta})$. In theory, the choice of $\mu(\ge \al)$ does
not affect the convergence of P1 method, and for the P2 method, the optimal convergence rate
holds only for $\mu\ge \al+1/2$. This is confirmed by Table \ref{tab:source-b-q-mu}: the
choice $\mu=\al+1/4$ fails to achieve the optimal order.

The numerical results for example (b2), i.e., $f(x)=(1-x)^{3/5}$, with $q(x)=x$, are shown in Table \ref{tab:source-right-sing}. In this
case, the weak solution singularity appears at both left and right end points. Like before we observe an
optimal convergence order $h^{\al}$ for the $P1$ FEM. Interestingly, for the $P2$ FEM, the empirical
orders are close to the theoretical ones when $\al$ is close to $1.5$, whose precise mechanism awaits
theoretical justification.

\begin{table}[hbt!]
\caption{The $L^2\II$-norm of the error for example (b1) with $q=x$, $\mu=4$, $\al=1.55, 1.75, 1.95$, $h=1/2^m$.}
\label{tab:source-b-q}
\vspace{-.3cm}
\begin{center}
     \begin{tabular}{|c|c|cccccc|c|}
     \hline
     $\al$  & $m$ &$3$ &$4$ &$5$ & $6$ & $7$ & $8$  &rate \\
     \hline
     $1.55$ & $P1$ &1.47e-2 &5.40e-3 &1.91e-3 &6.62e-4 &2.26e-4 &7.58e-5 & 1.52 (1.05) \\

     & $P2$   &2.21e-4 &3.88e-5 &6.71e-6 &1.15e-6 &1.98e-7 & 3.37e-8 & 2.54 (2.05)\\
    \hline
     $1.75$ & $P1$    &4.64e-3 &1.41e-3 &4.21e-4 &1.25e-4 &3.70e-5 & 1.08e-5 & 1.75 (1.25)\\

     & $P2$    & 3.35e-5 &5.05e-6 &7.56e-7 &1.13e-7 &1.68e-8 &2.52e-9 & 2.74 (2.25)\\
     \hline
     $1.95$ & $P1$     &1.64e-3 & 4.20e-4	&1.08e-4 &2.76e-5 &7.07e-6 &1.80e-6&  1.93 (1.45)\\

     & $P2$   &2.92e-6	&3.82e-7 &4.96e-8 &6.44e-9 &8.36e-10 &1.15e-10 &  2.95 (2.45)\\
     \hline
     \end{tabular}
\end{center}
\end{table}

\begin{table}[hbt!]
\caption{The $L^2\II$-norm of the error for example (b1) with $q=x$, $\mu=4$, $\al=1.05, 1.25, 1.45$, $h=1/2^m$.}
\label{tab:source-small_al}
\vspace{-.3cm}
\begin{center}
     \begin{tabular}{|c|c|cccccc|c|}
     \hline
     $\al$  & $m$ &$3$ &$4$ &$5$ & $6$ & $7$ & $8$  &rate \\
     \hline
     $1.05$ & $P1$ &5.13e-2 &3.12e-2 &1.73e-2 &8.97e-3 &4.41e-3 &2.06e-3 & 1.02 ($--$) \\

     & $P2$   &1.11e-2 &2.92e-3 &7.29e-4 &1.78e-4 &4.33e-5 & 1.03e-5 & 2.04 ($--$)\\
    \hline
     $1.25$ & $P1$    &2.05e-2 &1.01e-2 &4.61e-3 &2.01e-3 &8.49e-4 & 3.47e-4 & 1.24 ($--$)\\

     & $P2$    & 2.55e-3 &5.66e-4 &1.22e-4 &2.59e-5 &5.46e-6 &1.14e-6 & 2.25 ($--$)\\
     \hline
     $1.45$ & $P1$     &7.38e-3 & 2.90e-3	&1.10e-3 &4.10e-4 &1.50e-4 &5.40e-5&  1.43 ($--$)\\

     & $P2$   &5.19e-4	&9.85e-5 &1.83e-5 &3.38e-6 &6.20e-7 &1.13e-7 &  2.44 ($--$)\\
     \hline
     \end{tabular}
\end{center}
\end{table}

\begin{table}[hbt!]
\caption{The $L^2\II$-norm of the error for example (b1) with $q=x$, $\al=1.75$, $h=1/2^m$ and different $\mu$.}
\label{tab:source-b-q-mu}
\vspace{-.3cm}
\begin{center}
     \begin{tabular}{|c|c|cccccc|c|}
     \hline
     $\mu$  & $m$ &$3$ &$4$ &$5$ & $6$ & $7$ & $8$  &rate \\
     \hline
     $3$ & $P1$ &4.05e-3 &1.20e-3 &3.55e-4 &1.05e-4 &3.08e-5 &8.96e-6 & 1.75 (1.25) \\

     & $P2$   &2.21e-4 &3.88e-5 &6.71e-6 &1.15e-6 &1.98e-7 & 3.37e-8 &  2.74 (2.25)\\
    \hline
     $0.75$ & $P1$    &3.07e-3 &8.92e-4 &2.60e-4 &7.61e-5 &2.22e-5 & 6.41e-6 &  1.75 (1.25)\\

        & $P2$    & 3.35e-5 &5.05e-6 &7.56e-7 &1.13e-7 &1.68e-8 &2.52e-9 &  2.74 (2.25)\\
     \hline
     $2$ & $P1$     &3.57e-3 & 1.05e-3	&3.06e-4 &8.95e-5 &2.62e-5 &7.58e-6& 1.75 (1.25)\\

     & $P2$  &6.81e-6	&1.12e-6 &1.83e-7 &2.98e-8 &4.90e-9 &8.27e-10 &  2.60 ($--$)\\
     \hline
     \end{tabular}
\end{center}
\end{table}
\begin{table}[h!]
\caption{The $L^2\II$-norm of the error for  example (b2) with $q=x$, $\mu=3$, $\al=1.55, 1.75, 1.95$, $h=1/2^m$.}
\label{tab:source-right-sing}
\vspace{-.3cm}
\begin{center}
     \begin{tabular}{|c|c|cccccc|c|}
     \hline
     $\al$  & $m$ &$3$ &$4$ &$5$ & $6$ & $7$ & $8$  &rate \\
     \hline
     $1.55$ & $P1$ &5.15e-3 &1.72e-3 &5.74e-4 &1.91e-4 &6.38e-5 &2.12e-5 & 1.59 ($1.05$) \\

     & $P2$   &3.91e-5 &1.03e-5 &2.62e-6 &6.42e-7 &1.54e-7 & 3.59e-8 & 2.04 ($2.05$)\\
    \hline
     $1.75$ & $P1$    &1.98e-3 &5.54e-4 &1.55e-4 &4.39e-5 &1.24e-5 & 3.56e-6 & 1.82 ($1.25$)\\

     & $P2$    & 2.02e-5 &3.64e-6 &6.74e-7 &1.28e-7 &2.46e-8 &4.76e-9 & 2.38 ($2.25$)\\
     \hline
     $1.95$ & $P1$     &1.02e-3 & 2.59e-4	&6.52e-5 &1.65e-5 &4.15e-6 &1.04e-6 &  1.99 ($1.45$)\\

     & $P2$   &9.38e-6 & 1.27e-6	&1.73e-7 &2.34e-8 &3.18e-9 &4.33e-10 &  2.88 ($2.45$)\\
     \hline
     \end{tabular}
\end{center}
\end{table}

\subsubsection{Numerical results for example (c)}
Since the source term $f(x)=\chi_{[0,1/2]}$ is in $H^{\beta}\II$, $\beta \in (2-\alpha,1/2)$, by Theorem
\ref{thm:reg-new}, $w$ belongs to $H^{2+\beta}\II$. Hence by repeating the argument
for Theorem \ref{thm:error-u}, the P1 FEM achieves a convergence rate of $O(h^{\al-1+\beta})$, while that
 for the P2 FEM is $O(h^{\al -1/2+\beta})$, $\beta\in (2-\alpha,1/2)$. In Table \ref{tab:source-c-q},
we show the results when the discontinuous point is supported at a grid point. The P1 FEM converges at a
rate $O(h^{\al})$, which is one half order higher than the theoretical one. However, the P2 FEM exhibits
superconvergence, which is attributed to the fact that the solution is piecewise smooth
and $\| (w-w_h)' \|_{L^2}$ is second order convergent. In Table \ref{tab:source-c2-q}, we
show the error when the discontinuous point is not supported at a grid point. Then the empirical rate
for P2 FEM is $O(h^{\al+1/4})$, i.e., one quarter order higher than the theoretical ones.

\begin{table}[hbt!]
\caption{The $L^2\II$-norm of the error for example (c) with $q=x$, $\mu=4$,
$\al=1.55, 1.75, 1.95$, $h=1/2^m$.}
\label{tab:source-c-q}
\vspace{-.3cm}
\begin{center}
     \begin{tabular}{|c|c|cccccc|c|}
     \hline
     $\al$  & $m$ &$3$ &$4$ &$5$ & $6$ & $7$ & $8$  &rate \\
     \hline
     $1.55$ & $P1$ &4.40e-3 &1.54e-3 &5.33e-4 &1.83e-4 &6.22e-5 &2.09e-5 &  1.54 (1.05) \\

     & $P2$   &7.36e-5 &1.28e-5 &2.22e-6 &3.80e-7 &6.05e-8 & 1.11e-8 &  2.54 (2.05)\\
    \hline
     $1.75$ & $P1$    &1.84e-3 &5.18e-4 &1.46e-4 &4.17e-5 &1.20e-5 & 3.43e-6 &   1.81 (1.25)\\

     & $P2$    & 1.20e-5 &1.80e-6 &2.68e-7 &4.00e-8 &5.96e-9 &8.94e-10 &  2.74 (2.25)\\
     \hline
     $1.95$ & $P1$     &1.08e-3 & 2.72e-4	&6.87e-5 &1.73e-5 &4.36e-6 &1.09e-6&  1.99 (1.45)\\

     & $P2$   &1.14e-6	&1.49e-7 &1.94e-8 &2.51e-9 &3.26e-10 &4.51e-11 &  2.92 (2.45)\\
     \hline
     \end{tabular}
\end{center}
\end{table}

\begin{table}[hbt!]
\caption{The $L^2\II$-norm of the error for example (c) with $q=x$, $\mu=4$,
$\al=1.55, 1.75, 1.95$, $h=1/(2^m+1)$.}
\label{tab:source-c2-q}
\vspace{-.3cm}
\begin{center}
     \begin{tabular}{|c|c|cccccc|c|}
     \hline
     $\al$  & $m$ &$3$ &$4$ &$5$ & $6$ & $7$ & $8$  &rate \\
     \hline
     $1.55$ & $P1$ &1.43e-2 &5.65e-3 &2.08e-4 &7.37e-4 &2.55e-4 &8.60e-5 &  1.54 (1.05) \\

     & $P2$   &1.56e-4 &4.77e-5 &1.42e-5 &4.15e-6 &1.21e-6 & 3.49e-7 &  1.83 (1.55)\\
    \hline
     $1.75$ & $P1$    &4.47e-3 &1.48e-3 &4.65e-4 &1.42e-4 &4.23e-5 & 1.24e-5 &  1.76 (1.25)\\

     & $P2$    & 6.41e-5 &1.81e-5 &4.83e-6 &1.24e-6 &3.17e-7 &8.00e-8 &  2.00 (1.75)\\
     \hline
     $1.95$ & $P1$   &1.72e-3 &4.98e-4 &1.36e-4 &3.59e-5 &9.32e-6 &2.38e-6 &  1.96 (1.45)\\
            & $P2$   &2.98e-5 &7.34e-6 &1.71e-6 &3.84e-7 &8.51e-8 &1.97e-8 &  2.20 (1.95)\\
     \hline
     \end{tabular}
\end{center}
\end{table}

\subsection{Fractional Sturm-Liouville problem}
Now we illustrate the FSLP \eqref{eqn:eig} with the following potentials:
\begin{itemize}
  \item[(a)] a zero potential $q_1=0$;
  \item[(b)] a non-zero potential $q_2=x$.
\end{itemize}

Like before, we use a uniform mesh with a mesh size $h=1/(2^m\times10)$. We measure the
accuracy of an approximate eigenvalue $\la_h$ by the absolute error $|\la-\la_h|$ and the
approximate eigenfunction $u_h$ by the $L^2\II$-error $\|u-u_h\|_{L^2\II}$. It is well
known that problem \eqref{eqn:eig} with $q(x)=0$ has a countable number of eigenvalues $\la$
that are zeros of the Mittag-Leffler functions $E_{\al,\al}(-\lambda)$ \cite{Dzrbasjan:1970}
and the corresponding eigenfunction is given by $u(x)=x^{\al-1}E_{\al,\al}(-\la x^\al)$.
However, accurately computing zeros of the Mittag-Leffler function remains a challenging task and
it does not cover the interesting case of a general potential $q$. Thus we compute
eigenvalues $\la$ and eigenfunctions $u$ on a very refined mesh with $h=
1/6000$ by P2 FEM. The resulting discrete eigenvalue problems are solved by
built-in \texttt{MATLAB} function \texttt{eigs}.

The numerical results for the two potentials are presented in Tables \ref{tab:SL-q1-1.75}-\ref{tab:SL-q1-1.75-fun}
and \ref{tab:SL-q2-1.75}-\ref{tab:SL-q2-1.75-fun}, respectively, for $\alpha=1.75$. Although not presented,
we note that a similar convergence behavior is observed for other fractional orders. Since both $q_1$ and $q_2$ belong to
$H^1\II$, by Theorem \ref{thm:eig-approx}, the theoretical rate is $O(h^{\al+k-1/2})$, $k=0,1$, for
the approximate eigenvalues and eigenfunctions. The errors are identical for both potentials,
i.e., the potential term influences the errors very little. For $\alpha=1.75$, the first eight eigenvalues
are all real. Surprisingly, the approximation exhibits a second-order convergence for P1 method, and
the mechanism of superconvergence is to be analyzed.
Further, P2 approximation converges almost at rate of $O(h^{\al+1})$.
However, the eigenfunction approximation converges steadily at a standard rate $O(h^{\al+k})$.


\begin{table}[hbt!]
\centering
\caption{The absolute errors of the first eight eigenvalues, which are all real, for $\alpha=1.75$, $q_1$, $\mu=3$,
with mesh size $h=1/(10\times2^m)$.}\label{tab:SL-q1-1.75}
\vspace{-.2cm}
\begin{tabular}{cccccccccc}
\hline
  &$e\backslash m$ & 3 & 4 & 5 & 6 & 7 & 8 & rate\\
\hline
  &$\lambda_1$  & 1.73e-3 & 4.77e-4 & 1.33e-4 & 3.73e-5 & 1.05e-5 & 3.01e-6 &1.83\\
  &$\lambda_2$  & 1.15e-2 & 2.89e-3 & 7.30e-4 & 1.84e-4 & 4.68e-5 & 1.20e-5 &1.98\\
  &$\lambda_3$  & 5.34e-2 & 1.34e-2 & 3.39e-3 & 8.58e-4 & 2.18e-4 & 5.56e-5 &1.98\\
 P1 &$\lambda_4$  & 1.51e-1 & 3.76e-2 & 9.38e-3 & 2.34e-4 & 5.87e-4 & 1.47e-4 &2.00\\
  &$\lambda_5$  & 3.57e-1 & 8.92e-2 & 2.24e-2 & 5.61e-3 & 1.41e-3 & 3.56e-4 & 2.00\\
  &$\lambda_6$  & 6.89e-1 & 1.72e-1 & 4.28e-2 & 1.07e-2 & 2.66e-3 & 6.65e-4 & 2.01\\
  &$\lambda_7$  & 1.26e0  & 3.16e-1 & 7.91e-2 & 1.99e-2 & 4.99e-3 & 1.25e-3 & 2.00\\
  &$\lambda_8$  & 2.02e0  & 5.01e-1 & 1.25e-1 & 3.11e-2 & 7.75e-3 & 1.93e-3 & 2.01\\
\hline
  &$e\backslash m$ & 1 & 2 & 3 & 4 & 5 & 6 & rate\\
\hline
  &$\lambda_1$  & 1.00e-4 & 1.50e-5 & 2.22e-6 & 3.17e-7 & 3.36e-8 & 8.37e-9 &2.71\\
  &$\lambda_2$  & 1.57e-3 & 2.46e-4 & 3.72e-5 & 5.54e-6 & 8.07e-7 & 1.02e-7 &2.78\\
  &$\lambda_3$  & 5.69e-3 & 9.93e-4 & 1.57e-4 & 2.36e-5 & 3.49e-6 & 4.86e-7 &2.70\\
 P2 &$\lambda_4$  & 1.19e-2 & 2.55e-3 & 4.26e-4 & 6.60e-5 & 9.96e-6 & 1.49e-6 &2.59\\
  &$\lambda_5$  & 1.25e-2 & 4.77e-3 & 8.85e-4 & 1.41e-4 & 2.18e-5 & 3.39e-6 & 2.37\\
  &$\lambda_6$  & 5.52e-3 & 7.34e-3 & 1.59e-3 & 2.67e-4 & 4.12e-5 & 6.17e-6 & 2.61\\
  &$\lambda_7$  & 8.21e-2  & 7.92e-3 & 2.43e-3 & 4.37e-4 & 6.93e-5 & 1.03e-5 & 2.65\\
  &$\lambda_8$  & 2.39e-1  & 6.07e-3 & 3.52e-3 & 6.83e-4 & 1.11e-4 & 1.75e-5 & 2.64\\
\hline
\end{tabular}
\end{table}


\begin{table}[hbt!]
\centering
\caption{The $L^2\II$ errors of the first five eigenfunctions $u_i$, for $\alpha=1.75$, $q_1$, $\mu=3$,
with mesh size $h=1/(10\times2^m)$.}\label{tab:SL-q1-1.75-fun}
\vspace{-.2cm}
\begin{tabular}{cccccccccc}
\hline
  &$e\backslash m$ & 3 & 4 & 5 & 6 & 7 & 8 & rate\\
\hline
 & $u_1$ & 2.51e-4 & 7.48e-5 & 2.23e-5 & 6.66e-6 & 1.98e-6 & 5.91e-7 & 1.75\\
 & $u_2$ & 7.19e-4 & 2.11e-4 & 6.23e-5 & 1.84e-5 & 5.45e-6 & 1.62e-7 & 1.76\\
P1 & $u_3$ & 1.54e-3 & 4.49e-4 & 1.31e-4 & 3.86e-5 & 1.14e-5 & 3.36e-6 & 1.77\\
 & $u_4$ & 2.68e-3 & 7.73e-4 & 2.25e-4 & 6.57e-5 & 1.93e-5 & 5.68e-6 & 1.78\\
 & $u_5$ & 4.05e-3 & 1.16e-3 & 3.37e-4 & 9.81e-5 & 2.88e-5 & 8.46e-6 & 1.78\\
\hline
 & $e\backslash m$ & 1 & 2 & 3 & 4 & 5 & 6 & rate\\
\hline
 & $u_1$ & 5.39e-5 & 8.12e-6 & 1.22e-6 & 1.83e-7 & 2.72e-8 & 4.05e-9 & 2.74\\
 & $u_2$ & 4.01e-4 & 6.06e-5 & 9.11e-6 & 1.37e-6 & 2.04e-7 & 3.04e-8 & 2.74\\
P2 & $u_3$ & 1.22e-3 & 1.86e-4 & 2.80e-5 & 4.21e-6 & 6.30e-7 & 9.40e-8 & 2.73\\
 & $u_4$ & 2.68e-3 & 4.10e-4 & 6.21e-5 & 9.35e-6 & 1.40e-6 & 2.10e-7 & 2.73\\
 & $u_5$ & 4.87e-3 & 7.52e-4 & 1.14e-4 & 1.73e-5 & 2.59e-6 & 3.89e-7 & 2.73\\
\hline
\end{tabular}
\end{table}


\begin{table}[hbt!]
\centering
\caption{The absolute errors of the first eight eigenvalues, which are all real, for $\alpha=1.75$, $q_2$, $\mu=3$,
with mesh size $h=1/(10\times2^m)$.}\label{tab:SL-q2-1.75}
\vspace{-.2cm}
\begin{tabular}{cccccccccc}
\hline
  &$e\backslash m$ & 3 & 4 & 5 & 6 & 7 & 8 & rate\\
\hline
 & $\lambda_1$  & 1.69e-3 & 4.67e-4 & 1.30e-4 & 3.64e-5 & 1.02e-5 & 2.93e-6 &1.83\\
 & $\lambda_2$  & 1.11e-2 & 2.89e-3 & 7.29e-4 & 1.84e-4 & 4.68e-5 & 1.20e-5 &1.99\\
  &$\lambda_3$  & 5.34e-2 & 1.34e-2 & 3.39e-3 & 8.57e-4 & 2.17e-4 & 5.56e-5 &1.99\\
P1  &$\lambda_4$  & 1.51e-1 & 3.76e-2 & 9.38e-3 & 2.34e-4 & 5.87e-4 & 1.47e-4 &2.00\\
 & $\lambda_5$  & 3.56e-1 & 8.92e-2 & 2.24e-2 & 5.61e-3 & 1.41e-3 & 3.56e-4 & 2.00\\
  &$\lambda_6$  & 6.89e-1 & 1.72e-1 & 4.28e-2 & 1.07e-2 & 2.66e-3 & 6.65e-4 & 2.01\\
 & $\lambda_7$  & 1.26e0  & 3.16e-1 & 7.91e-2 & 2.00e-2 & 4.99e-3 & 1.25e-3 & 2.00\\
 & $\lambda_8$  & 2.02e0  & 5.01e-1 & 1.25e-1 & 3.11e-2 & 7.75e-3 & 1.93e-3 & 2.01\\
\hline
  &$e\backslash m$ & 1 & 2 & 3 & 4 & 5 & 6 & rate\\
\hline
  &$\lambda_1$  & 8.69e-4 & 1.30e-5 & 1.91e-6 & 2.64e-7 & 1.98e-8 & 1.65e-8 &2.71\\
 & $\lambda_2$  & 1.52e-3 & 2.38e-4 & 3.60e-5 & 5.36e-6 & 7.80e-7 & 9.80e-8 &2.78\\
 & $\lambda_3$  & 5.58e-3 & 9.76e-4 & 1.53e-4 & 2.32e-5 & 3.44e-6 & 4.74e-7 &2.77\\
 P2 & $\lambda_4$  & 1.17e-2 & 2.53e-3 & 4.22e-4 & 6.53e-5 & 9.86e-6 & 1.47e-6 &2.72\\
 & $\lambda_5$  & 1.22e-2 & 4.73e-3 & 8.78e-4 & 1.41e-4 & 2.17e-5 & 3.37e-6 & 2.68\\
 & $\lambda_6$  & 5.91e-2 & 7.28e-3 & 1.58e-3 & 2.64e-4 & 4.10e-5 & 6.14e-6 & 2.71\\
  &$\lambda_7$  & 8.26e-2  & 7.84e-3 & 2.41e-3 & 4.35e-4 & 6.90e-5 & 1.02e-5 & 2.66\\
 & $\lambda_8$  & 2.40e-1  & 5.97e-3 & 3.50e-3 & 6.80e-4 & 1.11e-4 & 1.75e-5 & 2.62\\
\hline
\end{tabular}
\end{table}
	 	

\begin{table}[hbt!]
\centering
\caption{The $L^2\II$ errors of the first five eigenfunctions $u_i$, for $\alpha=1.75$, $q_2$, $\mu=3$,
with mesh size $h=1/(10\times2^m)$.}\label{tab:SL-q2-1.75-fun}
\vspace{-.2cm}
\begin{tabular}{cccccccccc}
\hline
  &$e\backslash m$ & 3 & 4 & 5 & 6 & 7 & 8 & rate\\
\hline
 & $u_1$ & 2.49e-4 & 7.44e-5 & 2.22e-5 & 6.63e-6 & 1.98e-6 & 5.90e-7 & 1.75\\
 & $u_2$ & 7.27e-4 & 2.13e-4 & 6.29e-5 & 1.86e-5 & 5.50e-6 & 1.63e-7 & 1.76\\
P1 & $u_3$ & 1.55e-3 & 4.52e-4 & 1.32e-4 & 3.88e-5 & 1.14e-5 & 3.38e-6 & 1.77\\
 & $u_4$ & 2.70e-3 & 7.77e-4 & 2.26e-4 & 6.60e-5 & 1.94e-5 & 5.71e-6 & 1.77\\
 & $u_5$ & 4.07e-3 & 1.17e-3 & 3.38e-4 & 9.84e-5 & 2.88e-5 & 8.49e-6 & 1.78\\
\hline
 & $e\backslash m$ & 1 & 2 & 3 & 4 & 5 & 6 & rate\\
\hline
 & $u_1$ & 5.52e-5 & 8.34e-6 & 1.25e-6 & 1.88e-7 & 2.81e-8 & 4.21e-9 & 2.74\\
 & $u_2$ & 4.06e-4 & 6.13e-5 & 9.22e-6 & 1.38e-6 & 2.07e-7 & 3.08e-8 & 2.74\\
P2 & $u_3$ & 1.23e-3 & 1.87e-4 & 2.82e-5 & 4.24e-6 & 6.35e-7 & 9.48e-8 & 2.74\\
 & $u_4$ & 2.69e-3 & 4.12e-4 & 6.25e-5 & 9.41e-6 & 1.41e-6 & 2.11e-7 & 2.73\\
 & $u_5$ & 4.89e-3 & 7.56e-4 & 1.15e-4 & 1.73e-5 & 2.61e-6 & 3.90e-7 & 2.73\\
\hline
\end{tabular}
\end{table}

\subsection{Preconditioned algorithms}
One advantage of the new approach is that the leading term can naturally act as a preconditioner,
because it is dominant and has simple structure. We present the condition number of the systems
in Table \ref{tab:precondition}, in which P and W denotes with preconditioner and without
preconditioner, respectively. The system is more stable when $\al$ close to 2.
Interestingly, the preconditioned system is very stable for the choice $\mu=\al-1$,
which awaits theoretical justifications.

\begin{table}[hbt!]
\centering
\caption{condition number for P1, $q=x$, $\al=1.55, 1.75, 1.95$, $h=1/2^m$. (P - preconditioned, W - without preconditioner)
}\label{tab:precondition}
\vspace{-.2cm}
\begin{tabular}{c|ccccccc}
\hline
 $\al$ & $\mu$ & $m$& 3  & 5  & 7 & 9& 11\\
\hline
         &$\al-1$ & P& 1.21e0  & 1.74e0 & 5.89e0 & 5.81e1& 7.82e2\\
         &        & W& 2.32e1  & 3.80e2 & 6.08e3 & 9.73e4& 1.56e6\\
         \cline{2-8}
1.55        & $3$ & P& 2.10e0  & 1.12e1  & 1.34e2 & 1.85e3& 2.58e4\\
        &         & W& 3.30e2  & 5.47e2  & 8.78e3 & 1.41e5& 2.25e6\\
         \cline{2-8}
          &$4$    & P& 2.26e0  & 1.35e1  & 1.67e2 & 2.32e3 & 3.23e4\\
           &      & W& 3.41e1  & 5.71e2  & 9.18e3 & 1.47e5&  2.35e6\\
\hline
         &$\al-1$ & P& 1.10e0  & 1.19e0  & 1.53e0 & 3.08e0& 1.31e1\\
         &        & W& 2.36e1  & 3.87e2  & 6.20e3 & 9.91e4& 1.59e6\\
          \cline{2-8}
1.75        & $3$ & P& 1.39e0  & 2.72e0  & 1.09e1 & 7.44e1& 5.81e2\\
        &         & W& 2.85e1  & 4.69e2  & 7.51e3 & 1.20e5& 1.92e6\\
         \cline{2-8}
          & $4$   & P& 1.48e0  & 3.16e0  & 1.41e1 & 9.99e1& 7.86e2\\
        &         & W& 2.93e1  & 4.82e2  & 7.73e3 & 1.24e5& 1.98e6\\
\hline
         &$\al-1$ & P& 1.06e0  & 1.06e0  & 1.07e0 & 1.09e0& 1.17e0\\
         &        & W& 2.40e1  & 3.93e2  & 6.30e3 & 1.01e5& 1.61e6\\
          \cline{2-8}
1.95        & $3$ & P& 1.05e0  & 1.13e0  & 1.32e0 & 1.81e0& 3.38e0\\
        &         & W& 2.49e1  & 4.08e2  & 6.53e3 & 1.04e5& 1.67e6\\
         \cline{2-8}
         & $4$    & P& 1.06e0  & 1.16e0  & 1.40e0 & 2.05e0& 4.02e0\\
        &         & W& 2.50e1  & 4.10e2  & 6.57e3 & 1.05e5& 1.68e6\\
\hline
\end{tabular}
\end{table}

\section{Concluding remarks}

In this work, we have developed a new approach to the boundary value problem with a Riemann-Liouville
fractional derivative of order $\alpha\in(3/2,2)$ in the leading term. It is based on transforming the
problem into a second-order boundary value problem (possibly with nonlocal lower-order terms), and
eliminates several challenges with the classical formulation. The well-posedness of the formulation
and the regularity pickup were analyzed, and a novel Galerkin finite element method with P1 and P2
finite elements have been provided. The $L^2\II$ error estimate of the approximation has been established.
Further the approach was extended to the Sturm-Liouville problem, and convergence rates of the
eigenvalue and eigenfunction approximations were provided. Extensive numerical experiments were
provided to verify the convergence theory.

In our theoretical developments, the analysis is only for the case $\alpha>3/2$. The interesting
case $\alpha\in(1,3/2]$ was not covered by the theory. However, our numerical experiments indicate that
the approach converges equally well in this case. Further, the theoretical convergence rate is
one half order lower than the empirical one, for both source problem and Sturm-Liouville problem.
These gaps are still to be closed. Last, it is of much interest to extend the approach to the time
dependent case \cite{JinLazarovPasciakZhou:2014siam,JinLazarovPasciak:2015ima} as well as the multi-dimensional analogue,
for which a complete solution theory seems missing.

\section*{Acknowledgment}
The authors are grateful to the anonymous referees for their insightful comments, which have led
to improved presentation of the paper.
The research of R. Lazarov was supported in parts
by NSF Grant DMS-1016525 and also by Award No. KUS-C1-016-04, made by King Abdullah University of Science and Technology (KAUST).
X. Lu is supported by Natural Science Foundation of China No. 91230108 and No. 11471253. Z. Zhou was partially supported
by NSF Grant DMS-1016525.

\appendix
\section{Computation of the stiffness matrix}\label{app:stiff}

In this appendix we discuss the implementation of the new approach, especially the computation of
the stiffness matrix $A=[a_{ji}]$, with
\begin{equation*}
 a_{ji} = (\phi_i',\phi_j') + ({\DDR 0 {2-\alpha}} \phi_i,q\phi_j)+({\DDR 0 {2-\alpha}}\phi_i)(1) ( p,\phi_j),
\end{equation*}
with $\{\phi_i\}$ being the finite element basis functions. The computation of the leading
term $(\phi_i',\phi_j')$ is straightforward, and thus we focus on the last two terms.
Below we shall discuss the cases of piecewise linear and piecewise quadratic finite elements
separately.

\subsection{Piecewise linear finite elements}
To simplify the notation, we denote $\gamma=\alpha-1$. We first note the identity (with $h_i=x_i-x_{i-1}$)
\begin{equation*}
  \begin{aligned}
    {\DDR 0 {2-\alpha}}\phi_i(x) & = \frac{1}{\Gamma(\gamma)}\int_0^x (x-t)^{\gamma-1}\phi_i'(t)dt\\
      & = \frac{1}{\Gamma(\gamma)} \int_0^x (x-t)^{\gamma-1}(\frac{\chi_{[x_{i-1},x_i]}}{h_i}-\frac{\chi_{[x_{i},x_{i+1}]}}{h_{i+1}})dt\\
      & = \frac{1}{\Gamma(\gamma+1)}\left[h_i^{-1}((x-x_{i-1})_+^{\gamma} - (x-x_i)_+^{\gamma}) - h_{i+1}^{-1}((x-x_i)_+^{\gamma}-(x-x_{i+1})_+^{\gamma})\right],
  \end{aligned}
\end{equation*}
where $(c)_+$ denotes the positive part, i.e., $(c)_+=\max(c,0)$.
In the case of a uniform mesh, it simplifies to
\begin{equation*}
    {\DDR 0 {2-\alpha}}\phi_i(x) = \frac{1}{\Gamma(\gamma+1)h}
       \left((x-x_{i-1})_+^{\gamma} + (x-x_{i+1})_+^{\gamma} - 2(x-x_i)_+^{\gamma}\right).
\end{equation*}
Hence, the term $b_{ji}=\int_0^1 {\DDR 0 {2-\alpha}}\phi_i(x)q(x) \phi_j(x)dx$ in
the middle is of the form
\begin{equation*}
  \begin{aligned}
    b_{ji} & = \int_{x_{j-1}}^{x_j} q(x)\phi_j(x){\DDR 0 {2-\alpha}}\phi_i(x) dx
      + \int_{x_j}^{x_{j+1}} q(x)\phi_j(x){\DDR 0 {2-\alpha}}\phi_i(x)dx.
  \end{aligned}
\end{equation*}
The integrals on the right hand side can be evaluated accurately using an appropriate Gauss-Jacobi
quadrature rule. The last term is a rank-one matrix, and it requires only computing two
vectors. The quantity $ ({\DDR 0 {2-\alpha}}\phi_i)(1)$ can be computed in closed form
\begin{equation*}
  \begin{aligned}
  ({\DDR 0 {2-\alpha}}\phi_i)(1) & = \frac{1}{\Gamma(\gamma)}\int_0^1 (1-t)^{\gamma-1}\phi_i'(t)dt\\
  & = \frac{1}{\Gamma(\gamma)}\left[h_i^{-1}\int_{x_{i-1}}^{x_i}(1-t)^{\gamma-1}dt
   - h_{i+1}^{-1}\int_{x_i}^{x_{i+1}}(1-t)^{\gamma-1}dt\right]\\
   & = \frac{1}{\Gamma(\gamma+1)}\left[h_i^{-1}((1-x_{i-1})^{\gamma}-(1-x_{i})^\gamma)-h_{i+1}^{-1}((1-x_i)^\gamma-(1-x_{i+1})^{\gamma})\right].
  \end{aligned}
\end{equation*}
In case of a uniform mesh, it simplifies to
\begin{equation*}
   ({\DDR 0 {2-\alpha}}\phi_i)(1) = \frac{1}{\Gamma(\gamma+1)h}\left((1-x_{i-1})^{\gamma}+(1-x_{i+1})^{\gamma}-2(1-x_i)^{\gamma}\right).
\end{equation*}
For $h\ll x-x_i$, $(x-x_{i-1})_+^\gamma + (x-x_{i+1})_+^\gamma\approx 2(x-x_i)_+^\gamma$. Then
the expression for $\DDR0{2-\al} \phi_i(x)$ may suffer precision loss due to roundoff errors.
We may improve the accuracy by writing
$$  (x-x_{i-1})_+^\gamma - (x-x_{i})_+^\gamma =: A^\gamma-B^\gamma=B^\gamma\left[(A/B)^\ga-1\right]=B^\ga {\text{\bf expm1}}(\gamma\log(A/B)), $$
which allows stable computation in e.g., \texttt{MATLAB}.
Last, given $w_h$, one needs to recover $u_h$, which involves only fractional-order
differentiation of the basis $\{\phi_i\}$
\begin{equation*}
  u_h(x_j) = {\DDR 0 {2-\alpha}}w_h(x_j) - ({\DDR 0 {2-\alpha}}w_h)(1)x_j^\mu,
\end{equation*}
where the first term can be computed efficiently by (with $w_i=w_h(x_i)$)
\begin{equation*}
  \begin{aligned}
    u_h(x_j) &= \frac{1}{\Gamma(\gamma+1)}\sum_{i=1}^{j-1} w_i\left[h_i^{-1}((x_j-x_{i-1})^\gamma-(x_j-x_i)^\gamma)\right.\\
    &\qquad \left.+h_{i+1}^{-1}((x_j-x_i)^\gamma-(x_j-x_{i+1})^\gamma)\right]
    + \frac{1}{\Gamma(\gamma+1)}w_jh_j^{-1}(x_j-x_{j-1})^{\gamma}.
  \end{aligned}
\end{equation*}

\subsection{Piecewise quadratic finite elements}
Next we describe the case of piecewise quadratic finite elements, i.e.,
\begin{equation*}
  u = \sum_{i=1}^Nu_i\phi_i(x) + \sum_{i=0}^{N-1}u_{i'}\phi_{i'}(x),
\end{equation*}
where for simplicity, we denote by $x_{i'}=(x_i+x_{i+1})/2$, the middle point of the interval $[x_j,
x_{j+1}]$, and $\phi_{i'}$ denotes the basis function corresponding to the node $x_{i'}$. Then
like before, we find
\begin{equation*}
  \begin{aligned}
    {\DDR 0 {2-\alpha}}\phi_i(x) & = \frac{1}{\Gamma(\gamma)}\int_0^x (x-t)^{\gamma-1}\phi_i'(t)dt\\
      & = \frac{1}{\Gamma(\gamma)} \int_0^x (x-t)^{\gamma-1}(\frac{\chi_{[x_{i-1},x_i]}}{h_i}(3+4\frac{t-x_i}{h_i})+\frac{\chi_{[x_{i},x_{i+1}]}}{h_{i+1}}(-3+4\frac{t-x_i}{h_{i+1}}))dt\\
      & = \frac{1}{\Gamma(\gamma+1)} \left[3h_i^{-1}((x-x_{i-1})^\gamma_+-(x-x_i)^\gamma_+) -3h_{i+1}^{-1}((x-x_i)^\gamma_+-(x-x_{i+1})^\gamma_+)\right]\\
      & \qquad + \frac{1}{\Gamma(\gamma)}\left[ 4h_i^{-2}(\gamma^{-1}(x-x_i)((x-x_{i-1})^\gamma-(x-x_i)^\gamma)\right.\\
      & \qquad\qquad\left.-(\gamma+1)^{-1}((x-x_{i-1})^{\gamma+1}-(x-x_i)^{\gamma+1})\right]\\
      & \qquad + \frac{1}{\Gamma(\gamma)}\left[4h_{i+1}^{-2}(\gamma^{-1}(x-x_i)((x-x_i)^\gamma-(x-x_{i+1})^\gamma)\right.\\
      & \qquad\qquad\left.-(\gamma+1)^{-1}((x-x_i)^{\gamma+1}-(x-x_{i+1})^{\gamma+1}))\right].
  \end{aligned}
\end{equation*}
For a uniform mesh, the expression simplifies to
\begin{equation*}
  \begin{aligned}
    {\DDR 0 {2-\alpha}}\phi_i(x)
      & = \frac{3}{\Gamma(\gamma+1)}\left(3h^{-1}+4h^{-2}(x-x_i)\right)\left((x-x_{i-1})^\gamma_++(x-x_{i+1})_+^\gamma-2(x-x_i)^\gamma_+\right)\\
      &\qquad-\frac{4}{\Gamma(\gamma)(\gamma+1)h^2}((x-x_{i-1})^{\gamma+1}+(x-x_{i+1})^{\gamma+1}-2(x-x_i)^{\gamma+1}).
  \end{aligned}
\end{equation*}
Likewise, with $\phi_{i'}= 1 - 4\frac{(x-x_{i'})^2}{h_{i+1}^2}$, we have
\begin{equation*}
  \begin{aligned}
    {\DDR 0 {2-\alpha}}\phi_{i'}(x) & = \frac{1}{\Gamma(\gamma)}\int_0^x (x-t)^{\gamma-1}\phi_{i'}'(t)dt\\
      & = \frac{1}{\Gamma(\gamma)} \int_0^x -8h_{i+1}^{-2}(t-x_{i'})\chi_{[x_{i},x_{i+1}]}(x-t)^{\gamma-1}dt\\
      &= \frac{-8}{\Gamma(\gamma)h_{i+1}^2}\left[(\gamma+1)^{-1}((x-x_{i+1})_+^{\gamma+1}-(x-x_i)_+^{\gamma+1})\right.\\
      &\qquad \qquad \left.-\gamma^{-1}(x-x_{i'})((x-x_{i+1})_+^\gamma-(x-x_i)_+^\gamma)\right].
  \end{aligned}
\end{equation*}
The computation of the remaining terms is similar to the case of piecewise linear finite elements, and thus omitted.

\bibliographystyle{abbrv}
\bibliography{frac}
\end{document}